\documentclass{article}

\usepackage{arxiv}

\usepackage[utf8]{inputenc} % allow utf-8 input
\usepackage[T1]{fontenc}    % use 8-bit T1 fonts
\usepackage{url}            % simple URL typesetting
\usepackage{booktabs}       % professional-quality tables
\usepackage{amsfonts}       % blackboard math symbols
\usepackage{nicefrac}       % compact symbols for 1/2, etc.
\usepackage{microtype}      % microtypography
\usepackage{lipsum}
\usepackage{verbatim}       % Multiline comments

\RequirePackage{amssymb}    % Extra symbols
\RequirePackage{amsthm}     % Theorem-like environments
\RequirePackage{thmtools}   % Theorem-like environments, extends amsthm
\RequirePackage{mathtools}  % Fonts and environments for mathematical formulae
\RequirePackage{mathrsfs}   % Script font with \mathscr{}
\RequirePackage{cancel}     % Cancel terms with \cancel, \bcancel or \xcancel
\RequirePackage{stmaryrd}   % Brackets

\RequirePackage{tikz}             % Drawing tool
\usetikzlibrary{calc}
\usetikzlibrary{intersections}
\usetikzlibrary{decorations.markings}
\usetikzlibrary{cd} % An implementation of simple commutative diagrams in tikz
\RequirePackage[all]{xy} % For backwards compatibility with respect to publishing articles.

\RequirePackage{varioref}                                     % \vref
\RequirePackage{hyperref}                                     % Clickable links
\RequirePackage[nameinlink, capitalize, noabbrev]{cleveref}   % \cref
\RequirePackage{doi}           % Ignore LaTeX syntax in DOI links

\declaretheoremstyle[headfont   = \bfseries\sffamily,
                     notefont   = \normalfont,
                     spaceabove = 6pt plus 0pt minus 2pt]{plain}
\declaretheoremstyle[headfont   = \bfseries\sffamily,
                     notefont   = \normalfont,
                     spaceabove = 6pt plus 0pt minus 2pt]{definition}
\declaretheorem[style = plain, numberwithin = section]{theorem}
\declaretheorem[style = plain,      sibling = theorem]{corollary}
\declaretheorem[style = plain,      sibling = theorem]{lemma}
\declaretheorem[style = plain,      sibling = theorem]{proposition}

\declaretheorem[style = definition, sibling = theorem]{definition}
\declaretheorem[style = definition, sibling = theorem]{example}

\crefname{observation}{Observation}{Observations}
\Crefname{observation}{Observation}{Observations}
\crefname{conjecture}{Conjecture}{Conjectures}
\Crefname{conjecture}{Conjecture}{Conjectures}
\crefname{notation}{Notation}{Notations}
\Crefname{notation}{Notation}{Notations}
\crefname{paper}{Paper}{Papers}
\Crefname{paper}{Paper}{Papers}

\title{Optimal triangulation of regular simplicial sets}

\author{
  Vegard Fjellbo \\
  Department of Mathematics \\
  University of Oslo \\
  Oslo, Norway \\
  \texttt{vegard.fjellbo@gmail.com} \\
   \And
}

\begin{document}
\maketitle

\begin{abstract}
\noindent The Barratt nerve, denoted $B$, is the endofunctor that takes a simplicial set to the nerve of the poset of its non-degenerate simplices. The ordered simplicial complex $BSd\, X$, namely the Barratt nerve of the Kan subdivision $Sd\, X$, is a triangulation of the original simplicial set $X$ in the sense that there is a natural map $BSd\, X\to X$ whose geometric realization is homotopic to some homeomorphism. This is a refinement to the result that any simplicial set can be triangulated.

A simplicial set is said to be regular if each of its non-degenerate simplices is embedded along its $n$-th face. That $BSd\, X\to X$ is a triangulation of $X$ is a consequence of the fact that the Kan subdivision makes simplicial sets regular and that $BX$ is a triangulation of $X$ whenever $X$ is regular. In this paper, we argue that $B$, interpreted as a functor from regular to non-singular simplicial sets, is not just any triangulation, but in fact the best. We mean this in the sense that $B$ is the left Kan extension of barycentric subdivision along the Yoneda embedding.
\end{abstract}

% keywords can be removed

\section{Introduction}
\label{sec:intro}

Not every CW complex can be triangulated \cite{Me67}, but simplicial sets can. The latter fact is largely due to Barratt \cite{Ba56}, but a correct proof was first given by Fritsch and Puppe in \cite{FP67}. One can prove it by arguing that all regular CW complexes are trianguable, that regular simplicial sets give rise to regular CW complexes and that the geometric realization of the last vertex map $d_X:Sd\, X\to X$ \cite[§7]{Ka57}, from the Kan subdivision $Sd\, X$ of $X$ \cite[§7]{Ka57}, is homotopic to a homeomorphism. Fritsch and Piccinini \cite[pp.~208--209]{FP90} tell the whole story in detail.

By a regular simplicial set, we mean the following.
\begin{definition}
Let $X$ be a simplicial set and suppose $y$ a non-degenerate simplex, say of dimension $n$. The simplicial subset of $X$ generated by $y\delta _n$ is denoted $Y'$. We can then consider the diagram
\begin{displaymath}
\xymatrix@=1em{
\Delta [n-1] \ar[d]_{\delta _n} \ar[r] & Y' \ar[d] \ar@/^1.5pc/[ddr] \\
\Delta [n] \ar@/_1pc/[drr] \ar[r] & \Delta [n]\sqcup _{\Delta [n-1]}Y' \ar[dr] \\
&& X
}
\end{displaymath}
in $sSet$ in which the upper left hand square is cocartesian. We say that $y$ is \textbf{regular} \cite[p.~208]{FP90} if the canonical map from the pushout is degreewise injective.
\end{definition}
\noindent We say that a simplicial set is \textbf{regular} if its non-degenerate simplices are regular.

There is a refinement to the result that all simplicial sets can be triangulated, as explained by Fritsch and Piccinini \cite[Ex.~5--8,~pp.~219--220]{FP90}. The triangulation of a given regular CW-complex described in Theorem 3.4.1 in \cite{FP90}, which is the barycentric subdivision when the CW-complex is the geometric realization of a simplicial complex, can be adapted to the setting of simplicial sets. The adaptation is an endofunctor $B:sSet\to sSet$ of simplicial sets, which is in \cite[p.~35]{WJR13} referred to as the Barratt nerve.

Let $N:Cat\to sSet$ be the fully faithful nerve functor from small categories to simplicial sets. Let $X^\sharp$ be the partially ordered set (poset) of non-degenerate simplices of $X$ with $y\leq x$ when $y$ is a face of $x$. In general, a poset $(P,\leq )$ can be thought of as a small category in the following way. Let the objects be the elements of $P$ and let there be a morphism $p\to p'$ whenever $p\leq p'$. The full subcategory of $Cat$ whose objects are the ones that arise from posets, we denote $PoSet$. The poset $X^\sharp$ is in some sense the smallest simplex category of $X$. The simplicial set $BX=NX^\sharp$ is the \textbf{Barratt nerve} of $X$.

There is a canonical map
\[b_X:Sd\, X\to BX\]
as explained in \cite[p.~37]{WJR13}. It is natural and expresses the viewpoint that $Sd$ is the left Kan extension of barycentric subdivision of standard simplices along the Yoneda embedding \cite[X.3~(10)]{ML98}. By this viewpoint, even the Kan subdivision performs barycentric subdivision on standard simplices \cite[X.3~Cor.~3]{ML98} as the Yoneda embedding is in particular fully faithful. Moreover, the map $b_X$ is degreewise surjective in general \cite[Lem.~2.2.10, p.~38]{WJR13} and an isomorphism if and only if $X$ is non-singular \cite[Lem.~2.2.11, p.~38]{WJR13}.

The Yoneda lemma puts the $n$-simplices $x$, $n\geq 0$, of a simplicial set $X$ in a natural bijective correspondence $x\mapsto \bar{x}$ with the simplicial maps $\bar{x} :\Delta [n]\to X$. Here, $\Delta [n]$ denotes the standard $n$-simplex. We refer to $\bar{x}$ as the \textbf{representing map} of the simplex $x$.
\begin{definition}
A simplicial set is \textbf{non-singular} if the representing map of each of its non-degenerate simplices is degreewise injective. Otherwise it is said to be \textbf{singular}.
\end{definition}
\noindent The inclusion $U$ of the full subcategory $nsSet$ of non-singular simplicial sets admits a left adjoint $D:sSet\to nsSet$, which is called desingularization \cite[Rem.~2.2.12]{WJR13}.

The map $b_X$ factors through the unit $\eta _{Sd\, X}:Sd\, X\to UD(Sd\, X)$ of the adjunction $(D,U)$. This gives rise to a degreewise surjective map
\[t_X:DSd\, X\to BX\]
that is a bijection in degree $0$. As $\eta _{Sd\, X}$ is degreewise surjective, we obtain a natural transformation $t$ between functors $sSet\to nsSet$. Our main result is the following.
\begin{theorem}\label{thm:main_opt_triang}
The natural map $t_X:DSd\, X\to BX$ is an isomorphism whenever $X$ is regular.
\end{theorem}
\noindent We will begin the proof of our main result in \cref{sec:mapcyl}.

A notion referred to as the reduced mapping cylinder \cite[§2.4]{WJR13} appears in the proof of \cref{thm:main_opt_triang}. Let $\varphi :P\to R$ be an order-preserving function between posets. The nerve
\[M(N\varphi )=N(P\times [1]\sqcup _PR)\]
of the pushout in the category of posets of the diagram
\begin{equation}
\begin{gathered}
\xymatrix{
P \ar[d]_{i_0} \ar[r]^\varphi & R \\
P\times [1]
}
\end{gathered}
\end{equation}
is known as the \textbf{(backwards) reduced mapping cylinder} of $N\varphi$ \cite[Def.~2.4.4]{WJR13}. If we think of posets as small categories as above and use the nerve to yield a diagram in $sSet$, then we obtain the pushout $T(N\varphi )$ known as the \textbf{(backwards) topological mapping cylinder} together with a \textbf{cylinder reduction map} \cite[Def.~2.4.5]{WJR13}
\[cr:T(N\varphi )\to M(N\varphi ).\]
In \cite[§2.4]{WJR13} the reduced mapping cylinder is introduced in full generality, meaning for an arbitrary simplicial map and not just for the nerve of an order-preserving function between posets. We refer to that source for the general construction.

The cylinder reduction map gives rise to a canonical map
\[dcr:DT(N\varphi )\to M(N\varphi )\]
from the desingularized toplogical mapping cylinder. \cref{thm:main_opt_triang} relies upon the following result, as we explain in \cref{sec:mapcyl}.
\begin{theorem}\label{thm:barratt_nerve_rep_map_dcr_iso}
Let $X$ be a regular simplicial set. For each $n\geq 0$ and each $n$-simplex $y$, the canonical map
\[dcr:DT(B(\bar{y} )\xrightarrow{\cong } M(B(\bar{y} ))\]
is an isomorphism.
\end{theorem}
\noindent This result does not seem to follow easily from the theory of \cite[§§2.4--2.5]{WJR13}, although it can essentially be deduced from \cite[Cor.~2.5.7]{WJR13} that $dcr$ is degreewise surjective and although $dcr$ is easily seen to be a bijection in degree $0$.

\cref{thm:barratt_nerve_rep_map_dcr_iso} is a refinement to one of the statements of Lemma $2.5.6$ of \cite[p.~71]{WJR13}. In \cref{sec:comparison}, we discuss a result related to \cref{thm:barratt_nerve_rep_map_dcr_iso}, but whose proof is easier. Namely, \cref{prop:cones_vs_mapping_cylinders} says that the desingularization of the cone on $NP$ is the reduced mapping cylinder of the unique map $NP\to \Delta [0]$, for every poset $P$.

The intuition behind \cref{thm:main_opt_triang} is as follows. One can look at $X=Sd\, Y$ for $Y$ some slightly singular example such as when $Y$ is the result of collapsing some $(n-1)$-dimensional face of a standard $n$-simplex. Another example is the model $Y=\Delta [n]/\partial \Delta [n]$ of the $n$-sphere for $0\leq n\leq 2$. When $n=0$ and $n=1$, it is clear that $t_X$ is an isomorphism. However, an argument is required for the case when $n=2$. These computations are performed in \cite[Section 4]{Fj19-DN}. Simple, but representative examples point in the same direction, namely that $t_X$ seems to be an isomorphism whenever $X$ is the Kan subdivision of some simplicial set $Y$.

If one is tempted to ask whether $t_X$ is an isomorphism whenever $X$ is a Kan subdivision, then it is no great leap to ask whether $t_X$ is an isomorphism for every regular simplicial set $X$. The book ``Spaces of PL manifolds and categories of simple maps'' \cite[Rem.~2.2.12,~p.~40]{WJR13} asks precisely this question. Our main result is thus an affirmative answer. There is a close relationship between regular simplicial sets and the simplicial sets that are Kan subdivisions. In fact, the Kan subdivision of every simplicial set is regular \cite[Prop.~4.6.10, p.~208]{FP90}.

In \cref{sec:conseq}, we discuss consequences of our main result. We explain how \cref{thm:main_opt_triang} follows from \cref{thm:barratt_nerve_rep_map_dcr_iso}, in \cref{sec:mapcyl}. It seems fitting that we refer forward to the various parts of the proof of \cref{thm:barratt_nerve_rep_map_dcr_iso} from \cref{sec:mapcyl} instead of from this introduction, so this is what we will do. Each section of this paper that follows \cref{sec:conseq} is essentially part of the proof of \cref{thm:main_opt_triang} and of \cref{thm:barratt_nerve_rep_map_dcr_iso}, except \cref{sec:cones}. The latter presents \cref{prop:cones_vs_mapping_cylinders}, which is a result on cones. It can be viewed as related to \cref{thm:barratt_nerve_rep_map_dcr_iso}.

\section{Applications}
\label{sec:conseq}

In this section, we discuss consequences of \cref{thm:main_opt_triang}.

Interpret $B$ as a functor $sSet\to nsSet$. On the one hand we have the triangulation $BSd:sSet\to nsSet$ of simplicial sets that may seem ad hoc, but that is concrete. On the other hand, we have the functor $DSd^2$ with the same source and target as $BSd$. It is somewhat cryptic as there is no other description of $D$ than the one we gave in \cref{sec:intro}. However, the functor $DSd^2$ has good formal properties. \cref{thm:main_opt_triang} implies that the natural map
\[t_{Sd\, X}:DSd^2\, X\xrightarrow{\cong } BSd\, X\]
is an isomorphism.

The functor $I=BSd$ is already a homotopically good way of making simplicial sets non-singular. It is known from \cite[§2.5]{WJR13} as the \textbf{improvement functor} and plays a role in that book. When we say that the improvement functor is a triangulation, we mean that there is a natural map $UIX\xrightarrow{s_X} X$ whose geometric realization is homotopic to a homeomorphism from the ordered simplicial complex $\lvert UIX\rvert$ to the CW complex $\lvert X\rvert$. The map $s_X$ is particularly well behaved when $X$ is a \textbf{finite simplicial set}, meaning that $X$ is generated by finitely many simplices.

Actually, the functor $DSd^2$ is also a homotopically relevant construction. By the main theorem of \cite{Fj19-HTY}, it can be made into a left Quillen functor of a Quillen equivalence when $sSet$ is equipped with the standard model structure due to Quillen \cite{Qu67}. Hence, \cref{thm:main_opt_triang} merges two preexisting theories into one.
\begin{definition}\label{def:simple_map_calculating_dsd2}
Let $X$ and $Y$ be finite simplicial sets and let $f:X\to Y$ be a simplicial map. We say that $f$ is \textbf{simple} if the point inverse $\lvert f\rvert ^{-1}(p)$ is contractible for any $p\in \lvert Y\rvert$.
\end{definition}
\noindent The map $s_X$ is simple when $X$ is finite. For a thorough discussion of the construction $I$ and the map $s_X$, see sections $2.2$, $2.3$, $2.5$ and $3.4$ of \cite{WJR13}.

Let $\Delta$ denote the category whose objects are the totally ordered sets $[n]$, $n\geq 0$, and whose morphisms $[m]\to [n]$ are the functions $\alpha$ such that $\alpha (i)\leq \alpha (j)$ whenever $i\leq j$. We refer to the morphisms as \textbf{operators}. Suppose $T:\Delta \to nsSet$ the functor that takes $[n]$ to the barycentric subdivision $\Delta '[n]$ of the standard $n$-simplex. Furthermore, we let $\Upsilon :\Delta \to rsSet$ be the Yoneda embedding $[n]\mapsto \Delta [n]$, corestricted to the full subcategory $rsSet$ of $sSet$ whose objects are the regular simplicial sets. Then $Sd$ is the left Kan extension of $UT$ along $U\Upsilon$.

Two related consequences of \cref{thm:main_opt_triang} are \cref{cor:left_Kan_extension_twice_barycentric_along_Yoneda} and \cref{cor:left_Kan_extension_barycentric_along_Yoneda} below.
\begin{corollary}\label{cor:left_Kan_extension_twice_barycentric_along_Yoneda}
The improvement functor $I:sSet\to nsSet$ is the left Kan extension of $DSdUT$ along $U\Upsilon$.
\end{corollary}
\begin{proof}
Because $Sd$ is the left Kan extension of $UT$ along $U\Upsilon$ and because $DSd$ has a right adjoint, it follows that $DSd^2=DSd\circ Sd$ is the left Kan extension of $DSd\circ UT$ along $U\Upsilon$ \cite[X.5~Thm.~1]{ML98}. The result now follows from \cref{thm:main_opt_triang}.
\end{proof}
\noindent With regards to our second consequence, namely \cref{cor:left_Kan_extension_barycentric_along_Yoneda}, the proof is short and relatively straight forward. However, it refers to some results that, although known, do not seem readily available in the literature. Therefore, we choose to present these (basic) results here.

We begin with the following two results, which say that a product of regular simplicial sets is regular and that a simplicial subset of a regular simplicial set is again regular. An argument is presented for the former of the two.
\begin{lemma}\label{lem:simplicial_subset_of_regular}
Let $X$ be a regular simplicial set and $A$ some simplicial subset. Then $A$ is regular.
\end{lemma}
\begin{proposition}\label{prop:product_of_regular}
Let
\[X=\prod _{j\in J}{X_j}\]
be a product of regular simplicial sets $X_j$, $j\in J$. Then $X$ is regular.
\end{proposition}
\begin{proof}[Proof of \cref{prop:product_of_regular}.]
Suppose $y\in X_n^\sharp$. For each $j\in J$, let $Y_j'$ be the image of the composite
\[\Delta [n-1]\xrightarrow{\delta _n} \Delta [n]\xrightarrow{\bar{y} } X\xrightarrow{pr_j} X_j.\]
Then we obtain the diagram
\begin{equation}
\label{eq:first_diagram_proof_prop_product_of_regular}
\begin{gathered}
\xymatrix{
\Delta [n-1] \ar[d]_{\delta _n} \ar[r] & Y_j' \ar[d] \ar@/^1.5pc/[ddr] \\
\Delta [n] \ar[r] \ar@/_1.5pc/[drr] & \Delta [n]\sqcup _{\Delta [n-1]}Y_j' \ar@{-->}[dr] \\
&& X_j
}
\end{gathered}
\end{equation}
in $sSet$, in which the canonical map from the pushout $\Delta [n]\sqcup _{\Delta [n-1]}Y_j'$ is degreewise injective as $X_j$ is regular.

The diagrams (\ref{eq:first_diagram_proof_prop_product_of_regular}) can be combined into the diagram
\begin{displaymath}
\xymatrix{
\Delta [n-1] \ar[d]_{\delta _n} \ar[r] & \prod _{j\in J}{Y_j'} \ar[d] \ar@/^1.5pc/[ddr] \\
\Delta [n] \ar[r] \ar@/_2pc/[drr]_{\bar{y} } & \prod _{j\in J}{(\Delta [n]\sqcup _{\Delta [n-1]}Y_j')} \ar[dr] \\
&& \prod _{j\in J}{X_j}
}
\end{displaymath}
that can be expanded to
\begin{equation}
\label{eq:second_diagram_proof_prop_product_of_regular}
\begin{gathered}
\xymatrix@C=0.9em@R=1.2em{
\Delta [n-1] \ar[d]_{\delta _n} \ar[r] & Y' \ar[d] \ar[r] & \prod _{j\in J}{Y_j'} \ar[d] \ar@/^2.5pc/[dddr]  \\
\Delta [n] \ar[r] \ar@/_2pc/[ddrrr]_{\bar{y} } \ar@/_1pc/[drr] & \Delta [n]\sqcup _{\Delta [n-1]}Y' \ar@{-->}[r] & \Delta [n]\sqcup _{\Delta [n-1]}(\prod _{j\in J}{Y_j')}) \ar@{-->}[d] \\
&& \prod _{j\in J}{(\Delta [n]\sqcup _{\Delta [n-1]}Y_j')} \ar[dr] \\
&&& \prod _{j\in J}{X_j}
}
\end{gathered}
\end{equation}
if we factor
\[\Delta [n-1]\to \prod _{j\in J}{Y_j'}\]
as a degreewise surjective map $\Delta [n-1]\to Y'$ followed by an inclusion.

Notice that $Y'$ is identified with the simplicial subset of $X$ that is generated by $y\delta _n$. It follows that $y$ is a regular simplex if the map
\[\Delta [n]\sqcup _{\Delta [n-1]}Y'\to X\]
is degreewise injective. This is true if the composite
\[\Delta [n]\sqcup _{\Delta [n-1]}Y'\to \Delta [n]\sqcup _{\Delta [n-1]}(\prod _{j\in J}{Y_j')})\to \prod _{j\in J}{(\Delta [n]\sqcup _{\Delta [n-1]}Y_j')}\]
is degreewise injective.

Assume that $w$ and $w'$ are different simplices of $\Delta [n]\sqcup _{\Delta [n-1]}Y'$ of the same degree, say of degree $q\geq 0$. We will prove that $w\mapsto e$ and $w'\mapsto e'$ are sent to different simplices $e$ and $e'$ in $\prod _{j\in J}{(\Delta [n]\sqcup _{\Delta [n-1]}Y_j')}$. There are three cases. The simplices $w$ and $w'$ can both be in the image of $Y'\to \Delta [n]\sqcup _{\Delta [n-1]}Y'$. It is also possible that neither of them are. By symmetry, the third possibility is that $w$ is in the image of $Y'\to \Delta [n]\sqcup _{\Delta [n-1]}Y'$ and that $w'$ is not.

Suppose $z\mapsto w$ and $z'\mapsto w'$ for some $q$-simplices $z$ and $z'$ of $Y'$. Then $Y'\to \prod _{j\in J}{Y_j'}$ maps $z\mapsto c$ and $z'\mapsto c'$ where $c$ and $c'$ are different as this map is an inclusion. Finally, the map
\[\prod _{j\in J}{Y_j'}\to \prod _{j\in J}{(\Delta [n]\sqcup _{\Delta [n-1]}Y_j')}\]
is degreewise injective as each simplicial set $Y_j'$, $j\in J$, is regular. Therefore, we get that $c\mapsto e$ and $c'\mapsto e'$ for different simplices $e$ and $e'$ in $\prod _{j\in J}{(\Delta [n]\sqcup _{\Delta [n-1]}Y_j')}$.

If neither $w$ nor $w'$ is in the image of $Y'\to \Delta [n]\sqcup _{\Delta [n-1]}Y'$, then we assume $b\mapsto w$ and $b'\mapsto w'$ for $q$-simplices $b$ and $b'$ of $\Delta [n]$. Choose some $j\in J$. The composite
\[\Delta [n]\to \prod _{j\in J}{(\Delta [n]\sqcup _{\Delta [n-1]}Y_j')}\xrightarrow{pr_j} \Delta [n]\sqcup _{\Delta [n-1]}Y_j'\]
sends $b$ and $b'$ to different simplices in $\Delta [n]\sqcup _{\Delta [n-1]}Y_j'$ as neither $b$ nor $b'$ is in the image of $N\delta _n$. Consequently, the first half of the composite maps $b\mapsto e$ and $b'\mapsto e'$ for different simplices $e$ and $e'$ in $\prod _{j\in J}{(\Delta [n]\sqcup _{\Delta [n-1]}Y_j')}$.

For the third case, assume that $z\mapsto w$ for some simplex $z$ in $Y'$ and that $w'$ is not in the image of $Y'\to \Delta [n]\sqcup _{\Delta [n-1]}Y'$. Then there is some simplex $b$ in $\Delta [n]$ such that $b'\mapsto w'$. Choose some $j\in J$. Consider the composites
\[\Delta [n-1]\to Y'\to \prod _{j\in J}{Y_j'}\xrightarrow{pr_j} Y_j'\]
and
\[\Delta [n]\to \prod _{j\in J}{(\Delta [n]\sqcup _{\Delta [n-1]}Y_j')}\xrightarrow{pr_j} \Delta [n]\sqcup _{\Delta [n-1]}Y_j'.\]
The first is the upper horizontal map in the cocartesian square in the $j$-th diagram (\ref{eq:first_diagram_proof_prop_product_of_regular}). The second is its cobase change along $N\delta _n$. As $b$ is not in the image of $N\delta _n$, it follows that the second of the two composites sends $b'$ to some simplex in $Y_j'$ that is not in the image of $Y_j'\to \Delta [n]\sqcup _{\Delta [n-1]}Y_j'$. Because the square
\begin{displaymath}
\xymatrix{
\prod _{j\in J}{Y_j'} \ar[d] \ar[r]^(.55){pr_j} & Y_j' \ar[d] \\
\prod _{j\in J}{(\Delta [n]\sqcup _{\Delta [n-1]}Y_j')} \ar[r]_(.58){pr_j} & \Delta [n]\sqcup _{\Delta [n-1]}Y_j'
}
\end{displaymath}
commutes, we see from (\ref{eq:second_diagram_proof_prop_product_of_regular}) that the image under $Y'\to \prod _{j\in J}{Y_j'}$ of $z$ is sent by $\prod _{j\in J}{Y_j'}\to \prod _{j\in J}{(\Delta [n]\sqcup _{\Delta [n-1]}Y_j')}$ to some $e$ that is different from $e'$ where $b'\mapsto e'$ under $\Delta [n]\to \prod _{j\in J}{(\Delta [n]\sqcup _{\Delta [n-1]}Y_j')}$.
\end{proof}
\noindent The results \cref{lem:simplicial_subset_of_regular} and \cref{prop:product_of_regular} yields the regularization functor, which is constructed thus.

Let $rsSet$ denote the full subcategory of $sSet$ whose objects are the regular simplicial sets. Given a simplicial set $X$, index a product over the quotient maps $X\to Y$ whose target $Y$ is regular. The product has as its factors the targets $Y$. We obtain a regular simplicial set $RX$ defined as the image of
\[X\to \prod _{f:X\to Y}{Y}\]
given by $x\mapsto (f(x))_f$. We say that $RX$ is the \textbf{regularization of $X$}. As the epimorphisms of simplicial sets are precisely the degreewise surjective maps and as every quotient map is degreewise surjective, the map $X\to RX$ is initial among the maps whose source is $X$ and whose target is regular.

The initial map becomes the unit of an adjunction in which $R$ is left adjoint to the inclusion $U:rsSet\to sSet$. One can in other words construct $R$ precisely as $D$ is constructed in \cite[Rem.~2.2.12]{WJR13}, except that non-singular simplicial sets is replaced with regular simplicial sets.

To prove \cref{cor:left_Kan_extension_barycentric_along_Yoneda}, we will also use the following basic result concerning Kan extensions. Note that we recycle the symbol $R$ for the purpose of stating and proving \cref{lem:Kan_extension_along_composite}.
\begin{lemma}\label{lem:Kan_extension_along_composite}
Consider a diagram
\[\mathscr{D} \xleftarrow{R} \mathscr{C} \xleftarrow{K} \mathscr{M} \xrightarrow{T} \mathscr{A}\]
where $\mathscr{M}$ is a small category and where $\mathscr{A}$ is cocomplete. Suppose the left Kan extension $Lan_{RK}T$ of $T$ along $RK$ exists.

If $R$ is fully faithful and admits a left adjoint functor $L:\mathscr{D} \to \mathscr{C}$, then the composite
\[Lan_KT=Lan_{RK}T\circ R\]
is the left Kan extension of $T$ along $K$.
\end{lemma}
\noindent Here, we follow the notation of \cite[§X]{ML98} closely as we will refer to results from that section in the proof.

Unfortunately, it seems that the context of \cref{lem:Kan_extension_along_composite} becomes clearest when we temporarily let $R$ denote the right adjoint indicated in the formulation of the lemma, rather than regularization. Then $R$ signifies \emph{right} and $L$ signifies \emph{left}. In this way, the case of \cref{lem:Kan_extension_along_composite} stands out from case of \cite[X.5~Thm.~1]{ML98}. However, the confusion should only be momentarily.

We are ready to prove the lemma.
\begin{proof}[Proof of \cref{lem:Kan_extension_along_composite}.]
Note that the left Kan extension $Lan_KT$ of $T$ along $K$ exists because $\mathscr{M}$ is small and because $\mathscr{A}$ is cocomplete \cite[§X.3~Cor.~2]{ML98}. By \cite[Ex.~X.4.3]{ML98}, the left Kan extension $Lan_R(Lan_KT)$ of $Lan_KT$ along $R$ exists as the left Kan extension $Lan_{RK}T$ exists. Moreover, we have that
\[Lan_R(Lan_KT)=Lan_{RK}T\]
by the same exercise.

We have natural transformations
\[\epsilon _K:T\Rightarrow (Lan_KT)\circ K\]
and
\[\epsilon _R:Lan_KT\Rightarrow Lan_R(Lan_KT)\circ R\]
that come with the two of our three Kan extensions. Next, let $\delta _R$ be the inverse of the map
\[(Lan_KT)\circ LR\xRightarrow{\cong } Lan_KT\]
that arises from the counit of the pair $(L,R)$. The counit $\epsilon _c:LRc\xrightarrow{\cong } c$ is an isomorphism as $R$ is fully faithful \cite[§IV.3~Thm.~1]{ML98}.

There is a (unique) natural transformation
\[\sigma _R:Lan_{RK}T\Rightarrow (Lan_KT)\circ L\]
such that the triangle on the left hand side in
\begin{equation}
\label{eq:first_diagram_proof_lem_Kan_extension_along_composite}
\begin{gathered}
\xymatrix@=1em{
& (Lan_{RK}T)\circ R \ar@/^/@{=>}[ddr]^(.55)\sigma \ar@{=>}[dd]^{\sigma _RR} \\
Lan_KT \ar@{=>}[ur]^(.4){\epsilon _R} \ar@{=>}[dr]_(.4){\delta _R} \\
& (Lan_KT)\circ LR \ar@{=>}[r]_(.63)\cong & Lan_KT
}
\end{gathered}
\end{equation}
commutes. The right hand side triangle in (\ref{eq:first_diagram_proof_lem_Kan_extension_along_composite}) was formed simply by letting $\sigma$ be the composite. Because $R$ is fully faithful, the natural transformation $\epsilon _R$ is a natural isomorphism \cite[§X.3~Cor.~3]{ML98}. This implies that $\sigma$ is a natural isomorphism and hence that $(Lan_{RK}T)\circ R$ is the left Kan extension of $T$ along $K$.
\end{proof}
\noindent With \cref{lem:Kan_extension_along_composite}, we have every result that we will use to establish our second corollary of \cref{thm:main_opt_triang}.

Similarly to the first corollary, we obtain the following.
\begin{corollary}\label{cor:left_Kan_extension_barycentric_along_Yoneda}
The composite
\[rsSet\xrightarrow{U} sSet\xrightarrow{B} nsSet\]
is a left Kan extension of $T$ along $\Upsilon$.
\end{corollary}
\begin{proof}
Let $(R,U)$ be the pair consisting of regularization and the inclusion. Because $Sd$ is the left Kan extension of $UT$ along $U\Upsilon$, the functor $SdU$ is the left Kan extension of $UT$ along $\Upsilon$ by \cref{lem:Kan_extension_along_composite}. The functor $DSdU$ is the left Kan extension of $T\cong DUT$ \cite[§IV.3~Thm.~1]{ML98} along $\Upsilon$ \cite[§X.5~Thm.~1]{ML98}. Now our result follows from \cref{thm:main_opt_triang}.
\end{proof}

\section{Mapping cylinders}
\label{sec:mapcyl}

We aim to prove \cref{thm:main_opt_triang}, which says that natural map
\[t_X:DSd\, X\to BX\]
is an isomorphism when $X$ is regular. In this section, we will explain how \cref{thm:main_opt_triang} follows from \cref{thm:barratt_nerve_rep_map_dcr_iso}. At the end of this section, we will make forward references to the work of proving latter.

The skeleton filtration of an arbitrary simplicial set $X$ gives rise to the diagram
\begin{equation}
\label{eq:diagram_mapcyl_skeleton_filtration}
\begin{gathered}
\xymatrix{
DSd\, X^0 \ar[d]^{t_{X^0}} \ar[r] & DSd\, X^1 \ar[d]^{t_{X^1}} \ar[r] & \dots \ar[r] & DSd\, X^n \ar[d]^{t_{X^n}} \ar[r] & \dots \\
BX^0 \ar[r] & BX^1 \ar[r] & \dots \ar[r] & BX^n \ar[r] & \dots
}
\end{gathered}
\end{equation}
and if the vertical maps are all isomorphisms, then $t_X$ is. This is because $t_X$ arises from (\ref{eq:diagram_mapcyl_skeleton_filtration}) as the canonical map between sequential colimits. Next, we explain the latter statement.

Consider the nerve $N:Cat\to sSet$ and the inclusion $U:PoSet\to Cat$. We let the symbol $N$ denote the corestriction to $nsSet$ of the composite $N\circ U$, also. Furthermore, we let $U$ denote the inclusion $U:nsSet\to sSet$. Then $N\circ U=U\circ N$ by definition.

The functor $DSd$ is a left adjoint, so in particular it preserves $X$ viewed as the colimit of its skeleton filtration. Furthermore, the functor
\[(-)^\sharp :sSet\to PoSet\]
is cocontinous, as we explain shortly.

If the inclusion of a full subcategory into the surrounding category has a left adjoint, then we will refer to the subcategory as a \textbf{reflective} subcategory. We then refer to the left adjoint as a \textbf{reflector}. Relevant examples are the facts that $nsSet$ is a reflective subcategory of $sSet$ and that $PoSet$ is a reflective subcategory of $Cat$. Note that the terminology is not standard. Although the fullness assumption seems more common today than before, Mac Lane's notion \cite{ML98}, for example, does not include fullness as an assumption in his definition.

We will also make use of the dual notion. If the inclusion of a full subcategory into the surrounding category has a right adjoint, then we will refer to the subcategory as a \textbf{coreflective} subcategory. Knowing that a subcategory is reflective or coreflective has a bearing on the formation of limits and colimits in the subcategory, as we will point out when it becomes relevant.

The (full) inclusion $U:PoSet\to Cat$ admits a left adjoint $p:Cat\to PoSet$, so $PoSet$ is a reflective subcategory of $Cat$. Furthermore, let $c:sSet\to Cat$ be left adjoint to $N:Cat\to sSet$. Notice that the map $c(b_X)$ gives rise to the map
\[cSd\, X\xrightarrow{c(b_X)} cUBX\xrightarrow{id} cUN(X^\sharp )\xrightarrow{id} cNU(X^\sharp )\xrightarrow{\epsilon _{UX^\sharp }} UX^\sharp\]
that sends the object corresponding to $[x,(\iota )]$ to the object $x$. The $0$-simplex of $Sd\, X$ is here thought of as uniquely represented by a minimal pair $(x,\iota )$ where $x$ is a non-degenerate simplex of $X$ and where $\iota$ is the identity $[n_x]\to [n_x]$ where $n_x$ is the degree of the simplex $x$. The natural map $b_X:Sd\, X\to UBX$ sends the $0$-simplex represented by $(x,(\iota ))$ to the functor $[0]\to X^\sharp$ with $0\mapsto x$.
\begin{lemma}\label{lem:sharp_functor_preserves_colimits}
The functor $(-)^\sharp :sSet\to PoSet$ preserves colimits.
\end{lemma}
\begin{proof}
The map $cSd\, X\to UX^\sharp$ is full and bijective on objects. If we apply posetification $p:Cat\to PoSet$ to the natural map $cSd\, Y\to UY^\sharp$, then we get an isomorphism. This conclusion comes from knowing that $p$ is a reflector. Because $pcSd$ is left adjoint to $ExNU$, where $Ex$ is right adjoint to $Sd$, it follows that $(-)^\sharp$ preserves colimits.
\end{proof}
\noindent This concludes our argument that $(-)^\sharp$ is cocontinous.

The map $t_{X^0}$ is an isomorphism as $b_{X^0}:Sd(X^0)\to B(X^0)$ is, say because $X^0$ is non-singular. Note that the $n$-skeleton $X^n$ can be built from $X^{n-1}$ by successively attaching the non-degenerate $n$-simplices along their boundaries. This building process may be transfinite.
\begin{definition}\label{def:sequence_optriang}
Let $\mathscr{C}$ be a cocomplete category and $\lambda$ some ordinal. A cocontinous functor $Y:\lambda \to \mathscr{C}$ is a \textbf{$\lambda$-sequence} in $\mathscr{C}$. We often write the $\lambda$-sequence as
\[Y^{[0]}\to Y^{[1]}\to \cdots \to Y^{[\beta ]}\to \cdots \]
where $Y^{[\beta ]}=Y(\beta )$ for $\beta <\lambda$. The canonical map $Y^{[0]}\to colim_{\beta <\lambda }Y^{[\beta ]}$ is the \textbf{composition} of $Y$. By a \textbf{sequence} we mean a $\lambda$-sequence for some ordinal $\lambda$.
\end{definition}
\noindent When $\lambda <\aleph _0$ is finite, then the composition of a $\lambda$-sequence is simply the composite of the maps in the sequence.

In the case when one admits $\lambda >\aleph _0$, like we do, one often uses the adjective \emph{transfinite} to indicate this as the term \emph{sequence} usually refers to the notion of $\aleph _0$-sequence. However, we usually admit $\lambda >\aleph _0$ and prefer instead to point it out if the sequence in question is a $\aleph _0$-sequence, whenever it is relevant.

The following highly flexible notion \cite[Def.~10.2.1]{Hi03} will be useful.
\begin{definition}
Let $n$ be some non-negative integer. If a map $f:X\to X'$ is a composition of some sequence $Y$ such that each map $Y^{[\beta ]}\to Y^{[\beta +1]}$ in the sequence is a cobase change of the inclusion $\partial \Delta [n]\to \Delta [n]$, then we say that $f$ is a \textbf{relative $\{ \partial \Delta [n]\to \Delta [n]\}$-cell complex} and we say that $Y$ is a presentation of $f$ as a relative $\{ \partial \Delta [n]\to \Delta [n]\}$-cell complex.
\end{definition}
\noindent If $X$ is a simplicial set, then the inclusion $X^{n-1}\to X^n$ is a relative $\{ \partial \Delta [n]\to \Delta [n]\}$-cell complex. See \cite[Cor.~4.2.4~(ii)]{FP90} and \cite[Prop.~10.2.14]{Hi03}. We will use this fact in our problem reduction below, stated as \cref{lem:first_reduction}.

For the compatibility between sequences and colimits in the two categories $PoSet$ and $nsSet$, we will use the following result.
\begin{lemma}
\label{lem:inclusion_poset_nsset_preserves_sequential_colimits}
The functor $N:PoSet\to nsSet$ preserves colimits of sequences.
\end{lemma}
\begin{proof}
The functor $U:PoSet\to Cat$ preserves colimits of sequences \cite[p.~216]{Ra10}. So does $N:Cat\to sSet$, as is well known. By a standard argument, the inclusion $U:nsSet\to sSet$ also preserves colimits of sequences. See for example \cite[Lemma~5.1.2.]{Fj18}. Because $nsSet$ is a reflective subcategory of $sSet$, the counit of the adjunction $(D,U)$ is in general an isomorphism. As $N\circ U=U\circ N$, it follows that $N:PoSet\to nsSet$ preserves colimits of sequences.
\end{proof}
\noindent Remember the non-standard notion of sequence from \cref{def:sequence_optriang}.

By the naturality of $t_X$, because $(-)^\sharp$ is cocontinous by \cref{lem:sharp_functor_preserves_colimits} and because $N$ preserves colimits of sequences by \cref{lem:inclusion_poset_nsset_preserves_sequential_colimits}, it follows that $t_X$ arises from (\ref{eq:diagram_mapcyl_skeleton_filtration}) as a map of sequential colimits. Thus $t_X$ is an isomorphism if $t_{X^n}$ is an isomorphism for each $n\geq 0$.

For our first problem reduction we will also need the following terms, which have a connection with properties of the Barratt nerve.
\begin{definition}
Suppose $\mathscr{B}$ a small category. Let $\mathscr{A}$ be a subcategory of $\mathscr{B}$. We will say that $\mathscr{A}$ is a \textbf{(co)sieve} in $\mathscr{B}$ if whenever we have a morphism $b\to b'$ whose target (source) is an object of $\mathscr{A}$, then the morphism is itself a morphism of $\mathscr{A}$.
\end{definition}
\begin{lemma}\label{lem:first_reduction}
The natural map $t_X:DSd\, X\to BX$ is an isomorphism whenever $X$ is regular if it is an isomorphism for each regular $X$ that is generated by a single simplex.
\end{lemma}
\begin{proof}
We will use a double induction. Suppose $n>0$ such that $t_X$ is an isomorphism whenever the dimension of $X$ is strictly lower than $n$. This will be our outer induction hypothesis. It is satisfied for $n=1$.

As our inner induction hypothesis, suppose $\lambda >0$ an ordinal such that a regular simplicial set $X$ has the property that $t_X$ is an isomorphism whenever the inclusion $X^{n-1}\to X$ can be presented by some $\gamma$-sequence
\[X^{n-1}=Y^{[0]}\to Y^{[1]}\to \cdots \to Y^{[\beta ]}\to \cdots \]
with $\gamma <\lambda$ as a relative $\{ \partial \Delta [n]\to \Delta [n]\}$-cell complex. The hypothesis is satisfied for $\lambda =1$ by the outer induction hypothesis.

Suppose $X$ a regular simplicial set such that the inclusion $X^{n-1}\to X$ can be presented by some $\lambda$-sequence $Y:\lambda \to sSet$ a relative $\{ \partial \Delta [n]\to \Delta [n]\}$-cell complex.

The case when $\lambda$ is a limit ordinal is handled by the same argument as the one concerning (\ref{eq:diagram_mapcyl_skeleton_filtration}).

Consider the case when $\lambda =\beta +1$ is a successor ordinal. Then $Y^{[\beta ]}$ is the colimit of a $\beta$-sequence, so $t_{Y^{[\beta ]}}$ is an isomorphism by the inner induction hypothesis. We shift notation and write $X'=Y^{[\beta ]}$ and $X=Y^{[\beta +1]}$. Thus we study an attaching
\[X=\Delta [n]\sqcup _{\partial \Delta [n]}X',\]
meaning the regular simplicial set $X$ is built from $X'$ by attaching some non-degenerate $n$-simplex $x$.

In general, the Barratt nerve behaves badly when applied to pushouts, so we choose a different decomposition of $X$ that the Barratt nerve respects. The decomposition that we have in mind, which is used for the same purpose in the proof of \cite[Prop.~2.5.8]{WJR13}, does not depend on regularity, although $X$ is regular.

Let $Y$ denote the simplicial subset of $X$ that is generated by $x$, or in other words, the image of its representing map $\bar{x} :\Delta [n]\to X$. If we take the pullback $Y'$ along the inclusion $X'\to X$, then we get a diagram
\begin{displaymath}
\xymatrix{
\partial \Delta [n] \ar[dd] \ar@{-->}[dr] \ar[rr] && X' \ar[dd] \\
& Y' \ar[dd] \ar[ur] \\
\Delta [n] \ar[dr] \ar@{-}[r]^(.65){\bar{x} } & \ar[r] & X \\
& Y \ar[ur]
}
\end{displaymath}
that gives rise to a factorization
\[X\to Y\sqcup _{Y'}X'\to X\]
of the identity. Furthermore, the map $Y\sqcup _{Y'}X'\to X$ is degreewise injective. Hence the simplicial set $X$ can be viewed as the pushout $Y\sqcup _{Y'}X'$.

Inductively, we can assume that $t_{Y'}$ is an isomorphism, so we have the diagram
\begin{displaymath}
\xymatrix{
DSd\, Y \ar[d]_{t_Y} & DSd\, Y' \ar[l] \ar[d]_{t_{Y'}}^\cong \ar[r] & DSd\, X' \ar[d]_{t_{X'}}^\cong \\
B\, Y & B\, Y' \ar[l] \ar[r] & B\, X'
}
\end{displaymath}
giving rise to a map between pushouts in $nsSet$ that $t_X$ factors through, by naturality. In fact, the Barratt nerve preserves the pushout $Y\sqcup _{Y'}X'$ as we explain in the next paragraph.

The sharp functor $(-)^\sharp :sSet\to PoSet$ is cocontinous by \cref{lem:sharp_functor_preserves_colimits}, so
\[X^\sharp =Y^\sharp \sqcup _{(Y')^\sharp}(X')^\sharp.\]
Moreover, $(-)^\sharp$ turns degreewise injective maps into sieves. % by \cref{lem:sharp_creates_sieves}
This means that the square
\begin{displaymath}
\xymatrix{
	U((Y')^\sharp ) \ar[d] \ar[r] & U((X')^\sharp ) \ar[d] \\
	U(Y^\sharp ) \ar[r] & U(X^\sharp )
}
\end{displaymath}
is cocartesian in $Cat$ \cite[p.~315]{Th80}. It is readily checked that the latter cocartesian square is preserved by $N:Cat\to sSet$ \cite[p.~315]{Th80}. Thus the Barratt nerve $B:sSet\to sSet$ preserves the pushout $X=Y\sqcup _{Y'}X'$. It follows that $t_X$ is an isomorphism if $t_Y$ is.

Note that $Y$ is generated by an $n$-simplex, by definition. We shift back to the previous notation $Y^{[\beta ]}=X'$ and $Y^{[\beta +1]}=X$. Namely, we have proven that $t_{Y^{[\beta +1]}}$ is an isomorphism given that $t_{Y^{[\beta ]}}$ is, and given the assumption of \cref{lem:first_reduction} that $t_X$ is an isomorphism whenever $X$ is regular and generated by a single simplex. This concludes the inner induction.

Let $X$ be some regular simplicial set of dimension $n$, meaning $X=X^n$. It follows from the outer induction hypothesis that $t_{X^{n-1}}$ is an isomorphism. By the inner induction, we know that $t_{X^n}$ is an isomorphism. It follows from the considerations concerning (\ref{eq:diagram_mapcyl_skeleton_filtration}) that $t_X$ is an isomorphism for every regular simplicial set $X$ given the assumption of \cref{lem:first_reduction}. Namely, the combination of \cref{lem:sharp_functor_preserves_colimits} and \cref{lem:inclusion_poset_nsset_preserves_sequential_colimits} shows that $t_X$ arises as a map between colimits of sequences from (\ref{eq:diagram_mapcyl_skeleton_filtration}).
\end{proof}
\noindent The purpose of reducing the proof that $t_X$ is an isomorphism for regular $X$ to the case when $X$ is generated by a single simplex is that we can then take advantage of a technique due to Thomason \cite{Th80}. This technique will reduce our problem further to its technical core, similar to how the use of mapping cylinders can be used in problem reduction. In fact, mapping cylinders is a special case and they show up in our argument.

The following definition of Thomason's \cite{Th80} has been adjusted to suit our needs, but in the restricted context of posets it is equivalent to the original one.
\begin{definition}[Thomason]\label{def:dwyer_map}
Let $k:P\to Q$ be a functor between posets $P$ and $Q$. We will say that $k$ is a \textbf{Dwyer map} if it embeds $P$ as a sieve in $Q$ and if there is a factorization
\begin{equation}
\label{eq:diagram_def_dwyer_map}
\begin{gathered}
\xymatrix{
P \ar[dr]_i \ar[rr]^k && Q \\
& W \ar[ur]_j
}
\end{gathered}
\end{equation}
such that $j$ a cosieve and such that $i$ embeds $P$ is a coreflective subcategory of $W$.
\end{definition}
\noindent That $P$ is a coreflective subcategory is to say that $i$ admits a right adjoint $r:W\to P$. The unit $a\to ri(a)$ is then an isomorphism in the poset $W$, which implies that it is an identity as there is no isomorphism in a poset, except the identities. In other words, $r$ is automatically a retraction. In turn, we get that the counit $\epsilon _w$ is the identity for $w=i(a)$.

By \cref{lem:first_reduction} we are left with proving \cref{prop:t_X_iso_when_X_regular_gen_single_simplex} below, in order to deduce \cref{thm:main_opt_triang}. \cref{prop:t_X_iso_when_X_regular_gen_single_simplex} can be proven from \cref{thm:barratt_nerve_rep_map_dcr_iso} by induction on the degree of the non-degenerate simplex that generates $X$.

The induction step is handled by the following lemma, which reduces our problem to a problem involving mapping cylinders, namely \cref{thm:barratt_nerve_rep_map_dcr_iso}.
\begin{lemma}\label{lem:second_reduction}
Suppose $X$ a regular simplicial set that is generated by a non-degenerate $n$-simplex $x$. Let $y=x\delta _n$. Then $X$ is decomposed by a cocartesian square
\begin{displaymath}
\xymatrix{
\Delta [n-1] \ar[d]_{N\delta _n} \ar[r]^(.6){\bar{y} } & Y \ar[d] \\
\Delta [n] \ar[r]_{\bar{x} } & X
}
\end{displaymath}
in $sSet$. Assume that $t_Y$ is an isomorphism.

Denote $P=\Delta [n-1]^\sharp$ and $Q=\Delta [n]^\sharp$. The map $(N\delta _n)^\sharp$ has a factorization $P\to W\to Q$ that satisfies the condition of being a Dwyer map. The pushouts $W\sqcup _PY^\sharp$ and $Q\sqcup _PY^\sharp$ in $Cat$ are a posets, so $N(W\sqcup _PY^\sharp )$ and $N(Q\sqcup _PY^\sharp )$ are non-singular. Furthermore, \dots
\begin{enumerate}
\item{\dots the map $t_X:DSd\, X\to BX$ is an isomorphism if the canonical map
\[D(NQ\sqcup _{NP}N(Y^\sharp ))\to N(Q\sqcup _PY^\sharp )\]
is an isomorphism. Finally, \dots }
\item{\dots the map $D(NQ\sqcup _{NP}N(Y^\sharp ))\to N(Q\sqcup _PY^\sharp )$ is an isomorphism if
\[D(NW\sqcup _{NP}N(Y^\sharp ))\to N(W\sqcup _PY^\sharp )\]
is an isomorphism.}
\end{enumerate}
\end{lemma}
\noindent The proof of \cref{lem:second_reduction} is deferred to \cref{sec:red}.

What is the announced connection with mapping cylinders? We now explain this. The structure of $(N\delta _n)^\sharp :P\to Q$ as a Dwyer map that we refer to in \cref{lem:second_reduction} is the factorization
\begin{equation}
\label{eq:zeroth_diagram_proof_lem_second_reduction}
\begin{gathered}
\xymatrix{
P \ar[dr]_{i_0} \ar[rr]^{N(\delta _n)^\sharp } && Q \\
& W=P\times [1] \ar[ur]_\psi
}
\end{gathered}
\end{equation}
in which $\psi$ is defined as follows. The function $\psi$ sends the pair
\[(\mu :[m]\to [n-1],0)\]
to the composite
\[[m]\xrightarrow{\mu } [n-1]\xrightarrow{\delta _n} [n],\]
and the pair $(\mu :[m]\to [n-1],1)$ to the face operator
\[[m+1]\to [n]\]
given by $j\mapsto \mu (j)$ for $0\leq j\leq m$ and $m+1\mapsto n$.

Notice that there is only one object of $Q$ that is not in the image of $\psi$, namely the $n$-th vertex $\varepsilon _n:[0]\to [n]$. \cref{fig:Cosieve_in_subdivided_two_simplex_that_contains_second_edge} illustrates the simplicial subset $NW$ of $NQ=B(\Delta [n])$ in the case when $n=2$.

\begin{figure}
\centering
\begin{tikzpicture}

% Vertices of a triangle
\coordinate (2) at (90:3cm);
\coordinate (0) at (210:3cm);
\coordinate (1) at (-30:3cm);

% Nodes to mark vertices of a triangle
\node [above] at (2) {2};
\node [below left] at (0) {0};
\node [below right] at (1) {1};

% Draw line between the vertices 0,1 and 2 and name the midpoints
\draw (0.north east)--(2.south) coordinate[midway](02);
\draw [very thick] (0.north east)--(1.north west) coordinate[midway](01);
\draw (1.north west)--(2.south) coordinate[midway](12);

% Barycenter, could also be found by means of the intersection library of tikz
\coordinate (012) at (barycentric cs:0=1,1=1,2=1);

% Draws lines from vertices of original triangle to barycentre of original triangle and...
\draw [very thick] (0.north east)--(012) coordinate[midway](0<012);
\draw [very thick] (1.north west)--(012) coordinate[midway](1<012);
\draw (2.south)--(012) coordinate[midway](2<012);
\draw [very thick] (01)--(012) coordinate[midway](01<012);
\draw [very thick] (12)--(012) coordinate[midway](12<012); % Illustrate the cosieve
\draw [very thick] (02)--(012) coordinate[midway](02<012); % Illustrate the cosieve

% Illustrate the cosieve
\draw [very thick] (0.north east)--(02);
\draw [very thick] (1.north west)--(12);

\end{tikzpicture}
\caption{Nerve of the cosieve $W$}
\label{fig:Cosieve_in_subdivided_two_simplex_that_contains_second_edge}
\end{figure}
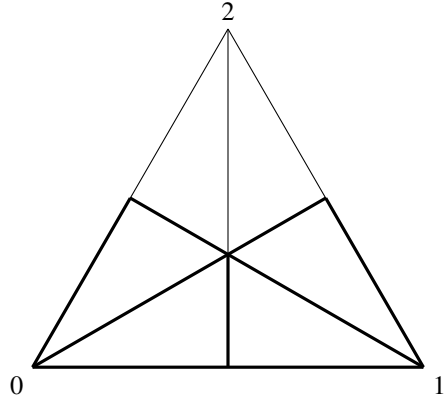

The pushout $Q\sqcup _PY^\sharp$ in $Cat$ is by the paragraph above taken along a Dwyer map, which implies that it is a poset \cite[Lem.~5.6.4]{Th80}. Furthermore, the pushout $W\sqcup _PY^\sharp$ in $Cat$ is a poset, say because it is taken along a rather trivial Dwyer map. Because $PoSet$ is a reflective subcategory of $Cat$ it follows that $W\sqcup _PY^\sharp$ can be considered a pushout in $PoSet$ of the underlying diagram.

Because $W=P\times [1]$, the pushout
\[T(B(\bar{y} ))=NW\sqcup _{NP}N(Y^\sharp )\]
in $sSet$ is the (backwards) topological mapping cylinder of $B(\bar{y} )$. Similarly,
\[M(B\bar{y} ))=N(W\sqcup _PY^\sharp )\]
is the (backwards) reduced mapping cylinder \cite[pp.~56--68]{WJR13}, which was defined in \cref{sec:intro}. Note that the canonical map
\[NW\sqcup _{NP}N(Y^\sharp )\to N(W\sqcup _PY^\sharp ),\]
is a guise of the cylinder reduction map $cr:T(B(\bar{y} ))\to M(B(\bar{y} ))$.

Next, consider the case when $X$ is generated by a single simplex. With the recognition made in the paragraph above, we are ready to discuss this case.
\begin{proposition}
\label{prop:t_X_iso_when_X_regular_gen_single_simplex}
Let $X$ be a regular simplicial set that is generated by a single simplex. Then $t_X$ is an isomorphism.
\end{proposition}
\begin{proof}
We will prove this by induction. Assume that $n>0$ is such that $t_X$ is an isomorphism for any regular $X$ that is generated by a non-degenerate simplex of degree $k<n$.

For the base step, one can note that the hypothesis holds for $n=1$ because $0$-dimensional simplicial sets are non-singular.

For the induction step, we assume that $X$ is as described in \cref{lem:second_reduction} and aim to prove that $t_X$ is an isomorphism. Notice that $Y$ is generated by the non-degenerate part of $y$, which is of degree $n-1$. This means that the assumption that $t_Y$ is an isomorphism, is justified.

\cref{lem:second_reduction} says that it suffices to prove that the map
\[D(NW\sqcup _{NP}N(Y^\sharp ))\to N(W\sqcup _PY^\sharp )\]
from Part $2$ is an isomorphism. In the text preceding this proof we saw that the latter map is a guise of the canonical map
\[dcr:DT(B(\bar{y} ))\to M(B(\bar{y} ))\]
whose source is the desingularized (backwards) topological mapping cylinder.

By \cref{thm:barratt_nerve_rep_map_dcr_iso}, the map $dcr$ is an isomorphism as $Y$ is regular.\\ \cref{lem:second_reduction} thus implies that $t_X$ is an isomorphism. This concludes the induction step.
\end{proof}
\noindent Note that \cref{prop:t_X_iso_when_X_regular_gen_single_simplex} relies upon \cref{thm:barratt_nerve_rep_map_dcr_iso}.

Now, recall \cref{lem:first_reduction}. We are ready to reduce \cref{thm:main_opt_triang} to \cref{thm:barratt_nerve_rep_map_dcr_iso}.
\begin{proof}[Proof of \cref{thm:main_opt_triang}.]
By \cref{prop:t_X_iso_when_X_regular_gen_single_simplex}, the assumption of \cref{lem:first_reduction} is satisfied. Thus we obtain \cref{thm:main_opt_triang}.
\end{proof}
\noindent Next, we keep our promise to explain the structure of the rest of this article.

Like the reader presumably have done so far, he preferably continues to read the sections in order, although there is a small detour in \cref{sec:cones}.

After \cref{sec:red}, which takes care of the deferred proof of \cref{lem:second_reduction}, we focus on \cref{thm:barratt_nerve_rep_map_dcr_iso} whose proof is rather technical. The work of proving \cref{thm:barratt_nerve_rep_map_dcr_iso} is divided into three tasks.

First, in \cref{sec:dzero}, we explain that
\[dcr:DT(B(\bar{y} ))\to M(B\bar{y} ))\]
is a bijection in degree $0$. This is a more or less formal argument involving not much more than the definition of the category $sSet$ as a set-valued functor category and the nerve functor.

Second, in \cref{sec:surject}, we show that $dcr$ is degreewise surjective. This is not trivial, however the answer is in our case more or less to be found in the pre-existing literature.

Third, in \cref{sec:zipping}, we do the part that seems hard to deduce from the literature, namely to prove that $dcr$ is degreewise injective in degrees above $0$. To do this, however, we separate out a few results in sections \ref{sec:tricat} and \ref{sec:deflation_thm}.

Finally, in \cref{sec:comparison}, we deduce \cref{thm:barratt_nerve_rep_map_dcr_iso} from the work of the three sections \ref{sec:dzero}, \ref{sec:surject} and \ref{sec:zipping}.

The reader may consider \cref{sec:cones} on cones as optional, as it is not really part of the storyline. On the other hand, it may yield insights into the idea behind the material in \cref{sec:zipping}. This is because the result presented in \cref{sec:cones} is a precursor. In addition, the reader may prefer our approach to the result stated as \cref{prop:cones_vs_mapping_cylinders} over any known proof.

\section{Reduction}
\label{sec:red}

This section is devoted to the proof of \cref{lem:second_reduction}. In the following proof we consider pushouts in four categories, namely the four objects in the commutative square
\begin{displaymath}
\xymatrix{
Cat \ar[r]^N & sSet \\
PoSet \ar[u]^U \ar[r]_N & nsSet \ar[u]_U
}
\end{displaymath}
of categories and functors.
\begin{proof}[Proof of \cref{lem:second_reduction} Part $1$.]
To factor the map $t_X$ in a useful way one can first factor $b_X:Sd\, X\to BX$ by means of the diagram
\begin{equation}
\label{eq:diagram_proof_lem_second_reduction}
\begin{gathered}
\xymatrix{
& Sd(\Delta [n]) \ar@{-}[d] \ar[dr] && Sd(\Delta [n-1]) \ar[ll]_{Sd(N\delta _n)} \ar@{-}[d] \ar[dr] \\
& \ar[d]_(.3)\cong ^(.3)b & Sd\, X \ar@/_6.3pc/[lldddd]_b \ar[dd]_(.65)f & \ar[d]_(.3)\cong ^(.3)b & Sd\, Y \ar[ll] \ar[dd] ^b \\
& B(\Delta [n]) \ar@/_/[lddd] \ar@/_/[dd] \ar[dr] & \ar[l] & B(\Delta [n-1]) \ar@{-}[l]_(.7){B(N\delta _n)} \ar[dr] \\
&& X' \ar[ld] && BY \ar@/^/[llld] \ar[ll] \ar@/^1pc/[lllldd] \\
& N(Q\sqcup _PY^\sharp ) \ar[ld] \\
BX
}
\end{gathered}
\end{equation}
where we have written the pushout $X'=NQ\sqcup _{NP}N(Y^\sharp )$ in $sSet$ of the lower square in the cube in (\ref{eq:diagram_proof_lem_second_reduction}) for brevity. The pushout $Q\sqcup _PY^\sharp$ is in $Cat$.

The functor $(-)^\sharp :sSet\to PoSet$ is cocontinous by \cref{lem:sharp_functor_preserves_colimits}. The pushout $Q\sqcup _PY^\sharp$ in $Cat$ is a poset \cite[Lem.~5.6.4]{Th80} as $P\to Q$ is a Dwyer map. Because $PoSet$ is a reflective subcategory of $Cat$ it then follows that the canonical map
\[Q\sqcup _PY^\sharp \xrightarrow{\cong } X^\sharp\]
is an isomorphism.

Naturality of $d_{Sd\, X}$ yields the diagram
\begin{displaymath}
\xymatrix{
Sd\, X \ar[d]_f \ar[r]^(.42)d & DSd(X) \ar[d]^{D(f)} \\
X' \ar[d]_{k} \ar[r]^(.47)d & DX' \ar@{-->}[ld]_(.53)l \ar[d]^{D(k)} \\
BX \ar[r]_(.45)d^(.45)\cong & DB(X)
}
\end{displaymath}
in which the diagonal map $l$ of the lower square arises due to the universal property of desingularization. It makes the upper left triangle of the lower square commute. Then the lower right triangle of the lower square commutes, also. This means we have a factorization of
\[b_X=k\circ f=l\circ d_{X'}\circ f=l\circ D(f)\circ d_{Sd\, X}\]
through $d_X$. The map $t_X$ is unique, so it follows that we get the useful factorization
\[t_X=l\circ D(f)\]
of the map $t_X$. The map $l$ is what we get when precomposing the canonical map
\[DX'\to N(Q \sqcup _PY^\sharp )\]
with the nerve of the canonical isomorphism
\[Q\sqcup _PY^\sharp \xrightarrow{\cong } X^\sharp .\]
Thus we see that $l$ is an isomorphism if $DX'\to N(Q \sqcup _PY^\sharp )$ is. We will see that $D(f)$ is an isomorphism, for formal reasons.

The map $D(f)$ is the canonical map between pushouts of $nsSet$ as $f$ is, by the universal property. It can be factored by applying the cocontinous functor $D$ to the diagram
\begin{displaymath}
\xymatrix@=1em{
Sd(\Delta [n-1]) \ar@/_5pc/[dddd]^b_\cong \ar[dd]_\cong ^d \ar[dr] \ar[rr] && Sd\, Y \ar@{-}[d] \ar[dr] \\
& Sd(\Delta [n]) \ar[dd]_(.6)\cong ^(.6)d \ar[rr] & \ar[d]^(.3)d & Sd\, X \ar[dd]_g \ar@/^2pc/[dddd]^f \\
DSd(\Delta [n-1]) \ar[dd]_\cong ^t \ar[dr] \ar@{-}[r] & \ar[r] & DSd\, Y \ar@{-}[d] \ar[dr] \\
& DSd(\Delta [n]) \ar[dd]_(.6)\cong ^(.6)t \ar[rr] & \ar[d]^(.3)t _(.3)\cong & X'' \ar[dd]_h \\
B(\Delta [n-1]) \ar[dr] \ar@{-}[r] & \ar[r] & BY \ar[dr] \\
& B(\Delta [n]) \ar[rr] && X'
}
\end{displaymath}
in $sSet$. The map $D(g)$ is an isomorphism because it is the canonical map between pushouts in $nsSet$ and because its source $DSd\, X$ and target $DX''$ are the most obvious ways of forming the pushout of the same diagram.

Recall from the formulation of the lemma that the map $t_Y$ is assumed to be an isomorphism. It follows that $D(h)$ is an isomorphism, hence $D(f)$ is an isomorphism. Hence, $t_X$ will be an isomorphism if $DX'\to N(Q\sqcup _PY^\sharp )$ is.
\end{proof}
\noindent We will conclude this section with the proof of Part $2$ of \cref{lem:second_reduction}.

The factorization $P\xrightarrow{i_0} W\xrightarrow{\psi } Q$ is through a cylinder $W=P\times [1]$. This coincidence means that we are dealing with mapping cylinders, although they play no explicit part in the rest of this section. What is relevant here, in the proof of Part $2$ of \cref{lem:second_reduction}, is the somewhat more general phenomenon of taking pushouts along the nerve of a Dwyer map.

As mapping cylinders are important technical tools it is an interesting problem in its own right to find interesting conditions under which the desingularized topological mapping cylinder is the reduced one. The work of \cref{sec:zipping} is a contribution to this end. When dealing with mapping cylinders of the nerve of a map between posets, Dwyer maps are always lurking in the background.

We are ready to prove Part $2$ of \cref{lem:second_reduction}, and thus completing the proof.
\begin{proof}[Proof of \cref{lem:second_reduction} Part $2$.]
The result follows immediately from \cref{prop:pushout_along_Dwyer} when we let
\begin{displaymath}
\begin{array}{rcl}
j\circ i & = & (N\delta _n)^\sharp \\
\varphi & = & (\bar{y} )^\sharp .
\end{array}
\end{displaymath}
In particular, $R=Y^\sharp$.
\end{proof}
\noindent Note that \cref{prop:pushout_along_Dwyer} slightly generalizes Part $2$ of \cref{lem:second_reduction}, but keeps the notation.

The next proposition is proven, essentially by using a technique by Thomason \cite[p.~316]{Th80} in his proof of Proposition 4.3 \cref{prop:pushout_along_Dwyer}.
\begin{proposition}\label{prop:pushout_along_Dwyer}
Let
\begin{displaymath}
\xymatrix{
NP \ar[d] \ar[r] & NR \ar[d] \\
NQ \ar[r] & NQ\sqcup _{NP}NR
}
\end{displaymath}
be a cocartesian square in $sSet$ where $P$, $Q$ and $R$ are posets and where $P\to Q$ is a Dwyer map with factorization $P\to W\to Q$. Then the map
\[D(NQ\sqcup _{NP}NR)\to N(Q\sqcup _PR)\]
is an isomorphism if
\[D(NW\sqcup _{NP}NR)\to N(W\sqcup _PR)\]
is an isomorphism.
\end{proposition}
\noindent By stating \cref{prop:pushout_along_Dwyer}, we have freed ourselves of the specific objects involved in \cref{lem:second_reduction}.

To tie together the studies of the two maps of \cref{prop:pushout_along_Dwyer} we consider the diagram
\begin{equation}
\label{eq:diagram_proof_prop_pushout_along_Dwyer}
\begin{gathered}
\xymatrix@C=0.5em@R=0.8em{
NP \ar[dd]^{Ni} \ar[dr] \ar[rr]^{N\varphi } && NR \ar@{-}[d] \ar[dr] \\
& NR \ar[dd] & \ar[d] & NR \ar[ll] \ar[dd] \\
NW \ar[dd]^{Nj} \ar[dr] \ar@{-}[r] & \ar[r] & NW\sqcup _{NP}NR \ar@{-}[d] \ar[dr]^\eta \\
& N(W\sqcup _PR) \ar[dd] & \ar[d] & D(NW\sqcup _{NP}NR) \ar[ll]_\zeta \ar[dd] \ar@/^6.5pc/[dddd] \\
NQ \ar[dr] \ar@{-}[r] & \ar[r] & NQ\sqcup _{NP}NR \ar@{-}@/_/[d] \ar[dr]^{\bar{\eta } } \\
& N(Q\sqcup _PR) & \ar@/_1pc/[ddr]^\eta & NQ\sqcup _{NW}D(NW\sqcup _{NP}NR) \ar[ll]_{\bar{\zeta }} \ar@{-->}[dd]_\xi \\
\\
&&& D(NQ\sqcup _{NP}NR) \ar@/^1pc/[lluu]^{\hat{\zeta } }
}
\end{gathered}
\end{equation}
in $sSet$. We take (\ref{eq:diagram_proof_prop_pushout_along_Dwyer}) as a naming scheme for the maps that play a role in the argument. Note that $\zeta$ is the map
\[dcr:DT(N\varphi )\to M(N\varphi )\]
in the case when $W=P\times [1]$ and when the map $i:P\to W$ is the map $p\mapsto (p,0)$.
\begin{proof}[Proof of \cref{prop:pushout_along_Dwyer}.]
By \cref{lem:proof_of_second_reduction}, the map $\hat{\zeta }$ is a cobase change in $sSet$ of $\zeta$. This means that $\hat{\zeta }$ is epic if $\zeta$ is. The epics of $sSet$ are precisely the degreewise surjective maps. Furthermore, a cobase change in $sSet$ of a degreewise injective map is again degreewise injective. This way we get that $\hat{\zeta }$ is an isomorphism if $\zeta$ is.
\end{proof}
\noindent Notice that \cref{prop:pushout_along_Dwyer} relies upon the following.
\begin{lemma}\label{lem:proof_of_second_reduction}
The map $\hat{\zeta }$ is a cobase change in $sSet$ of $\zeta$.
\end{lemma}
\begin{proof}
We will prove that $\hat{\zeta }$ is the cobase change in $sSet$ of $\zeta$ along
\[D(NW\sqcup _{NP}NR)\to D(NQ\sqcup _{NP}NR).\]
It suffices to prove that
\begin{equation}
\label{eq:first_diagram_lem_proof_of_second_reduction}
\begin{gathered}
\xymatrix{
NW \ar[d]_{Nj} \ar[r] & N(W\sqcup _PR) \ar[d] \\
NQ \ar[r] & N(Q\sqcup _PR)
}
\end{gathered}
\end{equation}
is cocartesian in $sSet$ and that $\xi$ is an isomorphism.

Let $V$ be the full subposet of $Q$ whose objects are those that are not in $P$. Then $V$ is a cosieve in $Q$ as $P$ is sieve. The square (\ref{eq:first_diagram_lem_proof_of_second_reduction}) fits into the bigger diagram
\begin{equation}
\label{eq:second_diagram_lem_proof_of_second_reduction}
\begin{gathered}
\xymatrix{
& NW \ar@{-}[d] \ar[dr] \\
NV\cap NW=N(V\cap W) \ar[dd] \ar[ur] \ar[rr] & \ar[d]^(.35){Nj} & N(W\sqcup _PR) \ar[dd] \\
& NQ \ar[dr] \\
NV \ar[rr] \ar[ur] && N(Q\sqcup _PR)
}
\end{gathered}
\end{equation}
where the cosieve $V$ in $Q$ makes an appearance.

The maps $V\cap W\to V$ and $V\cap W\to W$ are cosieves, so it follows that $Q$ can be decomposed as a pushout
\[Q\cong V\sqcup _{V\cap W}W\]
in $Cat$. Observe that $V\cap W\to W\sqcup _PR$ is also a cosieve. It follows that $N:Cat\to sSet$ preserves the pushouts $Q$ and
\[Q\sqcup _PR\cong V\sqcup _{V\cap W}(W\sqcup _PR).\]
From the diagram (\ref{eq:second_diagram_lem_proof_of_second_reduction}) we now see that (\ref{eq:first_diagram_lem_proof_of_second_reduction}) is cocartesian. From (\ref{eq:diagram_proof_prop_pushout_along_Dwyer}) we verify that $\bar{\zeta }$ is the cobase change in $sSet$ of $\zeta$ along
\[D(NW\sqcup _{NP}NR)\to NQ\sqcup _{NW}D(NW\sqcup _{NP}NR).\]
It remains to argue that $\xi$ is an isomorphism.

The nerve of the cosieve
\[V\cap W\to W\sqcup _PR\]
factors through
\[NV\cap NW\to D(NW\sqcup _{NP}NR),\]
so the latter is degreewise injective. Therefore
\[NQ\sqcup _{NW}D(NW\sqcup _{NP}NR)\cong NV\sqcup _{NV\cap NW}D(NW\sqcup _{NP}NR)\]
is non-singular.

The map
\[\eta :NQ\sqcup _{NP}NR\to D(NQ\sqcup _{NP}NR)\]
is degreewise surjective, therefore $\xi$ is. As the source of $\xi$ is non-singular, the map is an isomorphism.
\end{proof}

\section{Degree zero}
\label{sec:dzero}

We make use of the following result. Let $Cat$ denote the category of small categories.
\begin{lemma}\label{lem:degree_zero_colimit_category_nerve}
Let $F:J\to Cat$ be a functor whose source is a small category. Let $\mathscr{L}$ be the colimit of $F$. If $X$ is the colimit of the composite diagram
\[J\xrightarrow{F} Cat\xrightarrow{N} sSet,\]
then the canonical map $X\to N\mathscr{L}$ is a bijection in degree $0$.
\end{lemma}
\begin{proof}
Let $O$ denote the functor $Cat\to Set$ that takes a small category to the set of its objects. Recall that $O$ has a right adjoint, namely the functor that takes a set $S$ to the indiscrete category $IS$. This is the category whose set of objects is precisely $S$ and that is such that each hom set is a singleton.

We also use the functor
\[sSet=Fun(\Delta ^{op},Set)\xrightarrow{(-)_0} Set\]
that sends a simplicial set to the set of its $0$-simplices. There is a natural bijection
\[O\mathscr{C} \xrightarrow{\cong} (N\mathscr{C} )_0,\]
that takes an element $c$ of the set $O\mathscr{C}$ of objects of a small category $\mathscr{C}$ to the simplex $[0]\to \mathscr{C}$ with $0\mapsto c$.

Because $O$ is cocontinous, we get a canonical function $O\mathscr{L} \to X_0$. As colimits in $sSet$ are formed degreewise it follows that this function is a bijection. There is also a canonical function $O\mathscr{L} \to (N\mathscr{L} )_0$, which by naturality must be the mentioned bijection. The induced map $X_0\to (N\mathscr{L} )_0$ fits into a triangle
\begin{displaymath}
\xymatrix{
O\mathscr{L} \ar[dr]_\cong \ar[rr]^\cong && (N\mathscr{L} )_0 \\
& X_0 \ar[ur] \\
}
\end{displaymath}
that commutes by the universal property of the colimit $O\mathscr{L}$. Hence, our claim that $X\to N\mathscr{L}$ is a bijection in degree $0$ is true.
\end{proof}
\noindent An application of the previous lemma is the following example.
\begin{example}\label{ex:pushout_poset_along_dwyer}
Let $F':J\to PoSet$ be a diagram
\begin{displaymath}
\xymatrix{
P \ar[d]_k \ar[r]^\varphi & R \\
Q
}
\end{displaymath}
where $k$ is a Dwyer map. As $PoSet$ is a reflective subcategory of $Cat$, it follows that $U:PoSet\to Cat$ preserves the pushout of $F'$ \cite[Lem.~5.6.4]{Th80}. If $Q\sqcup _PR$ is the colimit of $F=U\circ F'$, then \cref{lem:degree_zero_colimit_category_nerve} says that
the canonical map
\[NQ\sqcup _{NP}NR\to N(Q\sqcup _PR)\]
is a bijection in degree $0$.

In particular, if $k$ is the special Dwyer map
\[k=i_0:P\to P\times [1]=Q,\]
then the reduction map
\[cr:T(N\varphi )\to M(N\varphi )\]
is in general a bijection in degree $0$.
\end{example}

\section{Tricategorical comparison}
\label{sec:tricat}

Often, one compares pushouts taken in several different subcategories. For example, in this article, we are interested in the commutative triangle
\begin{equation}
\label{eq:first_diagram_proof_prop_barratt_nerve_rep_map_dcr_inj}
\begin{gathered}
\xymatrix@=1em{
T(N\varphi ) \ar[dr]_{cr} \ar[rr]^{\eta } && DT(N\varphi ) \ar[ld]^{dcr} \\
& M(N\varphi )
}
\end{gathered}
\end{equation}
that factors the cylinder reduction map through the canonical degreewise surjective map $\eta$ whose target is the desingularization of the topological mapping cylinder.

To study $dcr$ is for many purposes to study $\eta$ and $cr$. There is a condition on
\[\eta _{T(N\varphi )}:T(N\varphi )\to DT(N\varphi )\]
that will ensure that $dcr$ is degreewise injective.
\begin{definition}
\label{def:siblings}
Whenever $x$ and $x'$ are simplices of the same degree of some simplicial set, we will say that they are \textbf{siblings} if $x\varepsilon _j=x'\varepsilon _j$ for all $j$.
\end{definition}
\noindent Our motivating example for the next result is $f=\eta _{T(N\varphi )}$, $g=dcr$ and $h=cr$.
\begin{proposition}\label{prop:criterion_degreewise_injective_into_nerve_computation}
Suppose we have a commutative diagram
\begin{displaymath}
\xymatrix{
X \ar[dr]_h \ar[rr]^f && Y \ar[ld]^g \\
& Z
}
\end{displaymath}
in $sSet$ in which $f$ is degreewise surjective and
\[h_0:X_0\to Z_0\]
is injective. Furthermore, assume that $Y$ is non-singular and that $Z$ is the nerve
of some poset. The simplicial map $g$ is injective in a given degree $q>0$ if and only if
\[f(x)=f(x')\]
whenever $x$ and $x'$ are embedded siblings of degree $q$.
\end{proposition}
\noindent Before we prove the proposition, we remind the reader of some standard piece of terminology.

Recall the Eilenberg-Zilber lemma \cite[Thm.~4.2.3]{FP90}, which says that each simplex $x$ of each simplicial set is uniquely a degeneration $x=x^\sharp x^\flat$ of a non-degenerate simplex. The non-degenerate simplex $x^\sharp$ is the \textbf{non-degenerate part} of $x$ and $x^\flat$ is the \textbf{degenerate part}.
\begin{proof}[Proof of \cref{prop:criterion_degreewise_injective_into_nerve_computation}.]
The ``only if'' part will not be needed, but we state it to emphasize that the conditions are equivalent under the hypothesis of the lemma. This part uses that the diagram commutes and that $Z$ is the nerve of a poset.

Suppose $g$ is injective in degree $q$ and that $x$ and $x'$ are siblings of degree $q$. Then
\[h(x)\varepsilon _j=h(x\varepsilon _j)=h(x'\varepsilon _j)=h(x')\varepsilon _j\]
for each $j$, so $h(x)$ and $h(x')$ are siblings. This implies that $h(x)=h(x')$ as $Z$ is the nerve of a poset. Because the diagram commutes and because $g$ is injective in degree $q$, it follows that $f(x)=f(x')$.

To prove the ``if'' part, we will use every condition of the hypothesis of the lemma, except that $Z$ is the nerve of a poset. First, observe that $g_0$ is injective as $h_0$ is injective and as $f_0$ is surjective and hence a bijection.

Suppose $f$ satisfies the described condition and that $y_1$ and $y_2$ are simplices of $Y$, of degree $q$, such that
\begin{equation}\label{lb:Equation0_criterion_embedded_sieblings}
g(y_1)=g(y_2).
\end{equation}
We prove that $y_1=y_2$, which will imply that $g$ is injective in degree $q$. This we do by proving that the non-degenerate parts and the degenerate parts of $y_1$ and $y_2$ are equal, respectively.

The two decompositions
\begin{displaymath}
g(y_1)=g(y_1)^\sharp g(y_1)^\flat
\end{displaymath}
\begin{displaymath}
g(y_1)=g(y_1^\sharp y_1^\flat )=g(y_1^\sharp )y_1^\flat =g(y_1^\sharp )^\sharp g(y_1^\sharp )^\flat y_1^\flat .
\end{displaymath}
are one and the same due to the uniqueness part of the Eilenberg-Zilber lemma.

As usual, then, we have the equations
\begin{equation}\label{lb:Equation1_criterion_embedded_sieblings}
g(y_1)^\sharp =g(y_1^\sharp )^\sharp
\end{equation}
\begin{equation}\label{lb:Equation2_criterion_embedded_sieblings}
g(y_1)^\flat =g(y_1^\sharp )^\flat y_1^\flat .
\end{equation}
However, because $Y$ is non-singular, the non-degenerate simplex $y_1^\sharp$ is embedded, which is the same as saying that its vertices are pairwise distinct. Because $g$ is injective in degree $0$ it follows that $g(y_1^\sharp )=g(y_1^\sharp )^\sharp$ is embedded and thus non-degenerate. This implies that (\ref{lb:Equation1_criterion_embedded_sieblings}) turns into
\begin{equation}\label{lb:Equation3_criterion_embedded_sieblings}
g(y_1)^\sharp =g(y_1^\sharp ).
\end{equation}
That $g(y_1^\sharp )$ is non-degenerate also implies that the degeneracy operator $g(y_1^\sharp )^\flat$ is the identity, meaning (\ref{lb:Equation2_criterion_embedded_sieblings}) turns into
\begin{equation}\label{lb:Equation4_criterion_embedded_sieblings}
g(y_1)^\flat =y_1^\flat .
\end{equation}
The reasoning we applied to $y_1$ is equally valid for $y_2$, so
\begin{equation}\label{lb:Equation5_criterion_embedded_sieblings}
g(y_2)^\sharp =g(y_2^\sharp )
\end{equation}
\begin{equation}\label{lb:Equation6_criterion_embedded_sieblings}
g(y_2)^\flat =y_2^\flat .
\end{equation}

Due to the assumption (\ref{lb:Equation0_criterion_embedded_sieblings}) the combination of (\ref{lb:Equation3_criterion_embedded_sieblings}) and (\ref{lb:Equation5_criterion_embedded_sieblings}) yields
\begin{equation}\label{lb:Equation7_criterion_embedded_sieblings}
g(y_1^\sharp )=g(y_2^\sharp )
\end{equation}
by the uniqueness part of the Eilenberg-Zilber lemma, again. For the same reason, the combination of (\ref{lb:Equation4_criterion_embedded_sieblings}) and (\ref{lb:Equation6_criterion_embedded_sieblings}) yields
\begin{equation}\label{lb:Equation8_criterion_embedded_sieblings}
y_1^\flat =y_2^\flat .
\end{equation}
Thus we get that the degenerate part of $y_1$ is equal to the degenerate part of $y_2$. It remains to prove that $y_1$ and $y_2$ have the same non-degenerate part.

Suppose $y_1^\sharp =f(x_1)$ and $y_2^\sharp =f(x_2)$. Such simplices $x_1$ and $x_2$ exist as $f$ is degreewise surjective, and they are embedded in $X$ as $y_1^\sharp$ and $y_2^\sharp$ are embedded in $Y$. Due to (\ref{lb:Equation7_criterion_embedded_sieblings}) we know that $h(x_1)=h(x_2)$, hence
\[h(x_1\varepsilon _j)=h(x_1)\varepsilon _j=h(x_2)\varepsilon _j=h(x_2\varepsilon _j)\]
for each $j$. As $h$ is injective in degree $0$ it follows that $x_1$ and $x_2$ are siblings. Finally, as $f$ sends embedded siblings to the same simplex, we get
\begin{equation}\label{lb:Equation9_criterion_embedded_sieblings}
y_1^\sharp =f(x_1)=f(x_2)=y_2^\sharp .
\end{equation}
Now we also know that the non-degenerate part of $y_1$ is equal to the non-degenerate part of $y_2$.

The equations (\ref{lb:Equation8_criterion_embedded_sieblings}) and (\ref{lb:Equation9_criterion_embedded_sieblings}) together imply that $y_1=y_2$, so it follows that $g$ is injective in degree $q$.
\end{proof}

\section{Concerning cones}
\label{sec:cones}

There is an interesting result concerning mapping cylinders that is related to \cref{thm:barratt_nerve_rep_map_dcr_iso}, namely \cref{prop:cones_vs_mapping_cylinders} below.

A possible proof of \cref{prop:cones_vs_mapping_cylinders} was an inspiration for \cref{thm:barratt_nerve_rep_map_dcr_iso}, so this section should also give the reader insight into the idea behind the proof of \cref{thm:barratt_nerve_rep_map_dcr_iso} and the proof by induction presented in \cref{sec:zipping}.

The result says the following.
\begin{proposition}
\label{prop:cones_vs_mapping_cylinders}
Let $P$ be some poset. Then the canonical map $dcr$ in the diagram
\begin{displaymath}
\xymatrix{
NP \ar[d]_{i_0} \ar[r]^{N\varphi } & \Delta [0] \ar[d] \ar@/^1.5pc/[ddr] \\
NP\times \Delta [1] \ar@/_1pc/[drr] \ar[r] & DT(N\varphi ) \ar[dr]^{dcr} \\
&& M(N\varphi )
}
\end{displaymath}
in $nsSet$ is an isomorphism.
\end{proposition}
\noindent In words, \cref{prop:cones_vs_mapping_cylinders} says that the desingularization of the cone on $NP$ is the reduced mapping cylinder of the unique map $NP\to \Delta [0]$.
\begin{proof}[Proof of \cref{prop:cones_vs_mapping_cylinders}.]
We will argue that $dcr$ is degreewise surjective, that it is a bijection in degree $0$ and finally that it is injective in degrees above $0$.

Let $k$ denote $i_0:P\to P\times[1]$ as in \cref{ex:pushout_poset_along_dwyer}. Then $k$ is canonically identified with $i_0:NP\to NP\times \Delta [1]$. Let $\bar{\varphi }$ denote the cobase change (in the category of posets) of $\varphi$ along $k$ and let $\bar{k}$ denote the cobase change of $k$ along $\varphi$. The map $k$ is a special kind of Dwyer map. Furthermore, let $r:P\times [1]\to P$ be the projection onto the first factor.

First, the map
\[cr:T(N\varphi )\to M(N\varphi )\]
is degreewise surjective in this special case, as we now explain. This immediately implies that $dcr$ is degreewise surjective.

If $z:[q]\to P\times [1]\sqcup _{P}[0]$ is some simplex in
\[M(N\varphi )=N(P\times [1]\sqcup _{P}[0]),\]
then there is some integer $j$ with $-1\leq j\leq q$ that has the property that $z(i)$ is in the image of $k$ for $i\leq j$ and that $z(i)$ is not in the image of $k$ for $i>j$. There is a $q$-simplex $x'$ of $T(N\varphi )$ whose image under $cr$ is $z$. It is defined thus.

If $j=q$, then we can simply define $x'$ as a degeneracy of the unique $0$-simplex that is in the image of $\Delta [0]\to T(N\varphi )$. Else if $j<q$, then we may for each $i>j$ define $x(i)$ as the uniqe element of $P\times [1]$ that $\bar{\varphi }$ sends to $z(i)$. Suppose
\[\bar{\varphi } (p,1)=z(j+1).\]
For each $i\leq j$, we define $x(i)=(p,1)$. Let $x'$ be the image of $x$ under $NP\times \Delta [1]\to T(N\varphi )$. It follows that $cr(x')=z$. This finishes the argument that $cr$ is degreewise surjective, and therefore that $dcr$ is. Keep in mind that $cr$ and $dcr$ fit into the commutative triangle (\ref{eq:first_diagram_proof_prop_barratt_nerve_rep_map_dcr_inj}).

By \cref{ex:pushout_poset_along_dwyer}, the map $cr$ is a bijection in degree $0$, which by (\ref{eq:first_diagram_proof_prop_barratt_nerve_rep_map_dcr_inj}) implies that $dcr$ is. It remains to verify that $dcr$ is injective in degrees above $0$.

For the argument that $dcr$ is degreewise injective in degrees above $0$, we will apply \cref{prop:criterion_degreewise_injective_into_nerve_computation} to (\ref{eq:first_diagram_proof_prop_barratt_nerve_rep_map_dcr_inj}).

Consider embedded siblings $x'$ and $y'$ of $T(N\varphi )$, say of degree $q>0$, whose zeroth common vertex is in the image of $\Delta [0]\to T(N\varphi )$ and whose $q$-th common vertex is not. This is the only non-trivial case. Let $x$ and $y$, respectively, be the unique simplices in $NP\times \Delta [1]$ whose image under $NP\times \Delta [1]\to T(N\varphi )$ is $x'$ and $y'$. Because the target of $\varphi$ has only one element, we see from (\ref{eq:first_diagram_proof_prop_cones_vs_mapping_cylinders}) that $\eta (x')=\eta (y')$.
\begin{equation}
\label{eq:first_diagram_proof_prop_cones_vs_mapping_cylinders}
\begin{gathered}
\xymatrix{
x(0) \ar[dr] \ar@{-->}[r] & kr(x(1)) \ar@{-->}[d] & y(0) \ar[ld] \ar@{-->}[l] \\
& x(1)=y(1) \ar[d] \\
& \dots \ar[d] \\
& x(q)=y(q)
}
\end{gathered}
\end{equation}
By \cref{prop:criterion_degreewise_injective_into_nerve_computation}, it follows that $dcr$ is injective in degree $q$. This finishes the proof that $dcr$ is injective in degrees above $0$ and hence an isomorphism.
\end{proof}

\section{Surjectivity of the cylinder reduction}
\label{sec:surject}

Not every cylinder reduction map
\[cr:T(N\varphi )\to M(N\varphi )\]
is degreewise surjective. It can happen that the dimension of the reduced mapping cylinder is strictly higher than the dimension of the topological mapping cylinder.
\begin{example}\label{ex:Non-surjective_cylinder_reduction}
Let $\varphi :P\to R$ be the functor between posets defined as follows. Its source is the poset
\[P=\{ b\leftarrow a\rightarrow c\}\]
and its target is the poset
\[R=\{ a'\rightarrow b'\rightarrow c'\} .\]
The functor is given on objects by $\varphi (a)=a'$, $\varphi (b)=b'$ and $\varphi (c)=c'$.

The (backwards) topological mapping cylinder $T(N\varphi )$ is evidently of dimension $2$. However, the (backwards) reduced mapping cylinder $M(N\varphi )$ is by definition the nerve of the pushout of the diagram
\begin{displaymath}
\xymatrix{
P \ar[d]_{i_0} \ar[r]^\varphi & R \\
P\times [1]
}
\end{displaymath}
in $PoSet$. Thus the reduced mapping cylinder is seen to be of dimension $3$, so the cylinder reduction map is not surjective in degree $3$.
\end{example}
\noindent Note that, in \cref{ex:Non-surjective_cylinder_reduction}, the image of $\varphi$, meaning the smallest subcategory of $R$ containing each object and each morphism hit by $\varphi$, is not a sieve in $R$. This is because the morphism $b'\to c'$ is not in the image of $\varphi$, though the object $c'$ is.

To take care of the surjectivity statement of \cref{thm:barratt_nerve_rep_map_dcr_iso}, we will adapt Lemma $2.5.6$ from \cite[p.~71]{WJR13} to our needs. Recall from \cref{def:simple_map_calculating_dsd2} the notion of simple maps. Note that a simple map is degreewise surjective. Simple maps are discussed in Chapter $2$ of \cite[pp.~29--97]{WJR13} and play a role in that book.

Let $f:X\to Y$ be a simplicial map whose source $X$ is a finite simplicial set. We say that $f$ is \textbf{simple onto its image} if the induced map $X\to f(X)$ is simple.
\begin{lemma}{(Lemma $2.5.6$ of \cite[p.~71]{WJR13})}
\label{lem:barrat_nerve_of_any_representing_map_of_regular_simple}
Let $X$ be a regular simplicial set. For each $n\geq 0$ and for each $n$-simplex $y$, the map
\[B(\bar{y} ):B(\Delta [n])\to BX\]
induced by the representing map $\bar{y}$ is simple onto its image.
\end{lemma}
\noindent Note that if $Y$ is the image of the representing map $\bar{y}$ of some simplex $y$, then $BY$ is the image of $B(\bar{y} )$ \cite[Lem.~2.4.20,~p.~67]{WJR13}.

In the rather lengthy proof of \cref{lem:barrat_nerve_of_any_representing_map_of_regular_simple}, which we display below, the following term from \cite[Def.~2.4.7]{WJR13} is an ingredient.
\begin{definition}
\label{def:simplicial_homotopy_equivalence_over_the_target}
Let $X$ and $Y$ be finite simplicial sets. A map $f:X\to Y$ is a \textbf{simplicial homotopy equivalence over the target} if there is a section $s:Y\to X$ of $f$ and a simplicial homotopy $H$ between $s\circ f$ and the identity $X\to X$ such that the square
\begin{displaymath}
\xymatrix{
X\times \Delta [1] \ar[d]_{pr_1} \ar[r]^(.6)H & X \ar[d]^f \\
X \ar[r]_f & Y
}
\end{displaymath}
commutes.
\end{definition}
\noindent Note that the homotopy $H$ provides a contraction of each point inverse of $\lvert f\rvert$, so $f$ is simple. There are several related notions that could fill the term of \cref{def:simplicial_homotopy_equivalence_over_the_target} \cite[p.~60]{WJR13} with meaning.

We are ready to prove the lemma.
\begin{proof}[Proof of \cref{lem:barrat_nerve_of_any_representing_map_of_regular_simple}.]
The proof is borrowed from the corresponding part of the proof of Lemma 2.5.6 from \cite[p.~71]{WJR13}. The only difference is that the notion of op-regularity is replaced with regularity.

Notice that it is enough to consider the representing maps of non-degenerate simplices. If $y$ is a simplex of $X$, say of degree $n$, then we can factor $B(\bar{y} )$ as
\[B(\Delta [n])\xrightarrow{B(Ny^\flat )} B(\Delta [k])\xrightarrow{B(\overline{y^\sharp } )} BX\]
where $k$ denotes the degree of $\bar{y}$ and where $B(Ny^\flat )$ is simple as it is a simplicial homotopy equivalence over the target.

Assume that $n>0$ is an integer such that the representing map of each non-degenerate simplex of $X$, of degree strictly less than $n$, is simple onto its image. Assume that $y$ is a non-degenerate simplex of degree $n$. We will prove that $B(\bar{y} )$ is simple onto its image.

Let $z=y\delta _n$ so that the image $Y$ of $\bar{y}$ is a pushout $\Delta [n]\sqcup _{\Delta [n-1]}Z$, where $Z$ is the image of $\bar{z} :\Delta [n-1]\to X$. Here, $\Delta [n]$ is attached to $Z$ along its $n$-th face, meaning along the map $N\delta _n$.

By the induction hypothesis, the map
\[B(\bar{z} ):B(\Delta [n-1])\to BX\]
is simple onto its image as the degree of $z^\sharp$ is at most $n-1$. The simplicial subset $BZ$ of $BX$ is the image of the Barratt nerve of the representing map of $z$ \cite[Lem.~2.4.20]{WJR13}.

In \cref{fig:Cosieve_in_subdivided_two_simplex_that_contains_second_edge} we displayed the simplicial set $B(\Delta [2])$ and highlighted a copy of $B(\Delta [1])\times \Delta [1]$ as a simplicial subset. The figure holds the key to a decomposition
\[B(\Delta [n])\cong M(B(\Delta [n-1])\to \Delta [0])\sqcup _{B(\Delta [n-1])}B(\Delta [n-1])\times \Delta [1]\]
as we now explain.

Recall the embedding $\psi :\Delta [n-1]^\sharp \times [1]\to \Delta [n]^\sharp$ from the proof of \cref{lem:second_reduction}. Form the backwards reduced mapping cylinder
\[M(B(\Delta [n-1])\to \Delta [0])\]
of $B(\Delta [n-1])\to \Delta [0]$. This mapping cylinder is the nerve of the pushout $P(\Delta [n-1]^\sharp \to [0])$ of
\begin{displaymath}
\xymatrix{
\Delta [n-1]^\sharp \ar[d]_{i_0} \ar[r] & [0] \\
\Delta [n-1]^\sharp \times [1]
}
\end{displaymath}
where $i_0$ takes $\mu$ to $(\mu ,0)$. The cosieve
\[i_1:\Delta [n-1]^\sharp \to \Delta [n-1]^\sharp \times [1]\]
gives rise to a cosieve
\[\Delta [n-1]^\sharp \to P(\Delta [n-1]^\sharp \to [0]).\]

Furthermore, we can define a map
\[\omega :\Delta [n-1]^\sharp \times [1]\to \Delta [n]^\sharp\]
by letting it send $(\mu ,0)$ to $\varepsilon _n$ and $(\mu :[m]\to [n-1],1)$ to the operator
\[[m+1]\to [n]\]
given by $j\mapsto \mu (j)$ for $0\leq j\leq m$ and $m+1\mapsto n$. From $\omega$ arises the right hand vertical map of the commutative square
\begin{displaymath}
\xymatrix{
\Delta [n-1]^\sharp \ar[d]_{i_1} \ar[r] & P(\Delta [n-1]^\sharp \to [0]) \ar[d] \\
\Delta [n-1]^\sharp \times [1] \ar[r]_(.6)\psi & \Delta [n]^\sharp
}
\end{displaymath}
which is cocartesian in the category of posets and even in the category of small categories. Moreover, the nerve functor preserves it as a cocartesian square as the legs are cosieves. This concludes the argument that $B(\Delta [n])$ can be decomposed as claimed.

Next, we display a suitable decomposition of $BY$. Form the backwards mapping cylinder $M(B(\bar{z} ))$ of the Barratt nerve of the corestriction to $Z$ of the representing map of the simplex $z$. Here, we overload the symbol $\bar{z}$. There is a degreewise injective map
\[B(\Delta [n-1])\xrightarrow{i_1} B(\Delta [n-1])\times \Delta [1]\to M(B(\bar{z} ))=NP((\bar{z} )^\sharp ),\]
which is induced by
\[\Delta [n-1]^\sharp \xrightarrow{i_1} \Delta [n-1]^\sharp \times [1]\to P((\bar{z} )^\sharp ).\]
As the simplicial set $Y$ is regular, the composite
\[P(\Delta [n-1]^\sharp \to [0])\to \Delta [n]^\sharp \xrightarrow{(\bar{y} )^\sharp } Y^\sharp\]
is injective on objects and actually a cosieve.

Next, consider the pushout
\[Y^\sharp =\Delta [n]^\sharp \sqcup _{\Delta [n-1]^\sharp }Z^\sharp .\]
Use the factorization of $(N\delta _n)^\sharp$ into $\psi \circ i_0$ as before and obtain $P((\bar{z} )^\sharp )\to Y^\sharp$ written as the cobase change of $\psi$ along $\Delta [n-1]^\sharp \times [1]\to P((\bar{z} )^\sharp )$. Combining this with the decomposition of $\Delta [n]^\sharp$ obtained above, we get the cocartesian square
\begin{displaymath}
\xymatrix{
\Delta [n-1]^\sharp \ar[d] \ar[r] & P(\Delta [n-1]^\sharp \to [0]) \ar[d] \\
P((\bar{z} )^\sharp ) \ar[r] & Y^\sharp
}
\end{displaymath}
which is also preserved by the nerve. Again, this is because both legs are cosieves. The diagram
\begin{displaymath}
\xymatrix{
B(\Delta [n-1])\times \Delta [1] \ar[d] & B(\Delta [n-1]) \ar[l]_(.4){i_1} \ar[d]^{id} \ar[r] & M(B(\Delta [n-1])\to \Delta [0]) \ar[d]^{id} \\
M(B(\bar{z} )) & B(\Delta [n-1]) \ar[l] \ar[r] & M(B(\Delta [n-1])\to \Delta [0])
}
\end{displaymath}
is a thus a way of obtaining the map $B(\Delta [n])\to BY$ induced by $B(\bar{y} )$.

On the cone $M(B(\Delta [n-1])\to \Delta [0])$, the map $B(\bar{y})$ is the identity. However, on the cylinder $B(\Delta [n-1])\times \Delta [1]$, the map $B(\bar{y})$ is the composite
\[B(\Delta [n-1])\times \Delta [1]\to T(B(\bar{z} ))\to M(B(\bar{z} )).\]

The first map of the composite above is the cobase change of the simple map $B(\bar{z} )$ along $i_0$. A point inverse of that map is either a point inverse under the induced map
\[\lvert B(\Delta [n-1])\rvert \times \lvert \Delta [1]\rvert -\lvert B(\Delta [n-1])\rvert \xrightarrow{\cong } \lvert T(B(\bar{z} )\rvert -\lvert BZ\rvert ,\]
which is a homeomorphism, or it can be considered a point inverse under
\[\lvert B(\bar{z} )\rvert :\lvert B(\Delta [n-1])\rvert \to  BZ.\]
Thus the first map of the composite is simple.

The second map is simple by the induction hypothesis and by Lemma 2.4.21. \cite[p.~67]{WJR13} as $\Delta [n-1]$ and $Z$ are of strictly lower dimension than $n$.
\end{proof}
\noindent Thus we obtain the technically important fact that for a regular simplicial set, the Barratt nerve of each representing map is simple onto its image.

We use the following notion from \cite[Def.~2.4.9]{WJR13}.
\begin{definition}
\label{def:simple_cylinder_reduction}
Let $\varphi :P\to R$ be a functor between finite posets $P$ and $R$. If the (backwards) cylinder reduction map
\[cr:T(N\varphi )\to M(N\varphi )\]
corresponding to the simplicial map $N\varphi $ is simple, then we say that $N\varphi $ has \textbf{simple cylinder reduction}.
\end{definition}
\noindent The notion of \cref{def:simple_cylinder_reduction} is defined more generally for a simplicial map $f:X\to Y$ whose source and target are both finite simplicial sets. However, we do not need the full generality.

Consider the following result, which is essentially Corollary $2.5.7$ from \cite[p.~71]{WJR13}.
\begin{proposition}\label{prop:map_between_regular_reduction_map_simple}
Let $X$ and $Y$ be finite regular simplicial sets. Suppose $f:X\to Y$ a simplicial map. Then $B(f)$ has simple cylinder reduction.
\end{proposition}
\begin{proof}
By \cref{lem:barrat_nerve_of_any_representing_map_of_regular_simple}, the map $B(\bar{x} )$ is simple onto its image for each $x\in X^\sharp$. Likewise for $Y$. Then $B(f)$ has simple cylinder reduction \cite[Lem.~2.4.21]{WJR13}.
\end{proof}

\section{A deflation theorem}
\label{sec:deflation_thm}

In this section, we will prove a basic yet useful result concerning regular simplicial sets.

We begin with the following observation.
\begin{lemma}\label{lem:property_of_regular_simplicial_set}
Let $y$ be a regular non-degenerate simplex, say of degree $n$, of some simplicial set. Assume that $y\mu$ and $y\nu$ are faces of $y$ such that the last vertex of $y$ is a vertex of one of them. If
\[(y\mu )^\sharp = (y\nu )^\sharp ,\]
then $\mu =\nu$.
\end{lemma}
\begin{proof}
Let $Y$ denote the simplicial subset that is generated by $y$ and let $Y'$ be generated by $y\delta _{n}$. Then the canonical map
\[\Delta [n]\sqcup _{\Delta [n-1]}Y'\xrightarrow{\cong } Y\]
is an isomorphism as $y$ is regular. We want to think of the simplices $y\mu$ and $y\nu$ of $Y$ as simplices of $\Delta [n]\sqcup _{\Delta [n-1]}Y'$.

Note that the isomorphism above implies that $y\varepsilon _n\neq y\varepsilon _j$ for all $j$ with $0\leq j<n$. By the assumption that the last vertex of $y$ is a vertex of $y\mu$ or of $y\nu$ we have that $n$ is in the image of at least one of the face operators $\mu$ and $\nu$. Say that $n$ is in the image of $\mu$. Then $y\mu =(y\mu )^\sharp$, and $y\mu$ is not in the image of
\[Y'\to \Delta [n]\sqcup _{\Delta [n-1]}Y'.\]
From $(y\mu )^\sharp =(y\nu )^\sharp$ it follows that $(y\nu )^\sharp$ is not in the image of this map, hence $y\nu$ is not. As $y\nu$ is the image of $\nu$ under
\[\Delta [n]\to \Delta [n]\sqcup _{\Delta [n-1]}Y'\]
it follows that $\nu$ is not in the image of $N\delta _n$, hence $n$ is in the image of $\nu$. This means that $y\nu =(y\nu )^\sharp$. Now it follows that $y\mu =y\nu$, so $\mu$ and $\nu$ must have the same source, say $[k]$. The function
\[\Delta [n]_k\to (\Delta [n]\sqcup _{\Delta [n-1]}Y')_k\]
is injective on the complement of the image of $(N\delta _n)_k$, which implies
\[\mu =\nu .\]
\end{proof}
\noindent Now, \cref{lem:property_of_regular_simplicial_set} may be intuitively obvious. However, the next result may not be obvious.

Consider a $2$-simplex of some regular simplicial set such that the non-degenerate parts of the first face and the second face are equal. Then the $2$-simplex is degenerate. Moreover, its non-degenerate part is equal to the two previously mentioned non-degenerate parts. In this sense, the $2$-simplex is deflated. One can say the following, in general.
\begin{proposition}\label{prop:deflation_theorem}
Let $X$ be a regular simplicial set and $y$ a simplex, say of degree $n$. Suppose $[n]$ the union of the images of two face operators $\mu$ and $\nu$ and that neither image is contained in the other. If
\[(y\mu )^\sharp =(y\nu )^\sharp ,\]
then $y$ is degenerate with non-degenerate part equal to the non-degenerate parts of $y\mu$ and $y\nu$.
\end{proposition}
\begin{proof}
Note that \cref{lem:property_of_regular_simplicial_set} immediately implies that $y$ is degenerate. Now, define
\[\alpha =y^\flat \mu\]
and take the unique factorization of
\[\alpha =\alpha ^\sharp \alpha ^\flat\]
into a degeneracy operator $\alpha ^\flat$ followed by a face operator $\alpha ^\sharp$. Similarly, we write
\[y^\flat \nu =\beta =\beta ^\sharp \beta ^\flat .\]
Now, the union of the images of the face operators $\alpha ^\sharp$ and $\beta ^\sharp$ is equal to their common target as the pair $(\mu ,\nu )$ has this property.

The left hand side of the equation $(y\mu )^\sharp =(y\nu )^\sharp$ can be written
\[(y^\sharp y^\flat \mu )^\sharp =(y^\sharp \alpha ^\sharp \alpha ^\flat )^\sharp =(y^\sharp \alpha ^\sharp )^\sharp\]
and the right hand side can be written
\[(y^\sharp y^\flat \nu )^\sharp =(y^\sharp \beta ^\sharp \beta ^\flat )^\sharp =(y^\sharp \beta ^\sharp )^\sharp.\]
By \cref{lem:property_of_regular_simplicial_set}, it follows that $\alpha ^\sharp =\beta ^\sharp$. As the union of the images of $\alpha ^\sharp$ and $\beta ^\sharp$ is equal to their common target it follows that both of the face operators are equal to the identity. This means that
\[(y^\sharp \alpha ^\sharp )^\sharp =(y^\sharp )^\sharp =y^\sharp\]
and the leftmost expression is equal to $(y\mu )^\sharp$. This concludes the proof.
\end{proof}

\section{Zipping}
\label{sec:zipping}

The canonical map
\[dcr:DT(N\varphi )\to M(N\varphi )\]
from the desingularized topological mapping cylinder to the reduced one is not necessarily degreewise injective.
\begin{example}\label{ex:Non-injective_dcylinder_reduction}
Let $f:\Delta [1]\to \Delta [1]/\partial \Delta [1]$ be the canonical map whose source is the standard $1$-simplex and whose target is the simplicial set one gets by taking the standard $1$-simplex and then identifying the zeroth and the first vertex.

The desingularized (backwards) topological mapping cylinder $DT(B(f))$ has two distinct non-degenerate $2$-simplices that are siblings. Thus
\[dcr:DT(B(f))\to M(B(f))\]
is not injective in degree $2$. In fact, $dcr$ fails to be injective even in degree $1$.
\end{example}
\noindent Note that $\Delta [1]/\partial \Delta [1]$ is not regular.

Compare the following proposition with \cref{thm:barratt_nerve_rep_map_dcr_iso}.
\begin{proposition}\label{prop:barratt_nerve_rep_map_dcr_inj}
Let $X$ be a regular simplicial set and $r$ some simplex of $X$, say of degree $n$. The canonical map
\[dcr:DT(B(\bar{r} ))\to M(B(\bar{r} ))\]
is injective in each positive degree.
\end{proposition}
\noindent The use of the letter $r$ instead of the letter $y$ as in \cref{thm:barratt_nerve_rep_map_dcr_iso} is a shift in notation that is meant to contribute to readability in the argument below. To prove \cref{prop:barratt_nerve_rep_map_dcr_inj}, we will let $\varphi =(\bar{r} )^\sharp$ and apply \cref{prop:criterion_degreewise_injective_into_nerve_computation} to the diagram (\ref{eq:first_diagram_proof_prop_barratt_nerve_rep_map_dcr_inj}).

As before, we write $P=\Delta [n]^\sharp$, $R=X^\sharp$ and $W=P\times [1]$. The reason we use the letter $W$ to denote $P\times [1]$ is that we at a later point will think of $P\times [1]$ as embedded in $Q=\Delta [n+1]^\sharp$ like in (\ref{eq:zeroth_diagram_proof_lem_second_reduction}) except that $n$ is replaced by $n+1$.

We study pushouts in $sSet$ and $nsSet$ of the diagram
\begin{equation}
\label{eq:second_diagram_proof_prop_barratt_nerve_rep_map_dcr_inj}
\begin{gathered}
\xymatrix{
NP \ar[d]_{k=Ni_0} \ar[r]^{f=N\varphi} & NR \\
NW
}
\end{gathered}
\end{equation}
and we study the canonical map
\[\eta :T(f)\to DT(f)\]
between them. The letter $k$ is not needed in the same capacity as in (\ref{eq:diagram_def_dwyer_map}). Instead its meaning is explained by (\ref{eq:second_diagram_proof_prop_barratt_nerve_rep_map_dcr_inj}). The notation is thus close to the one in the triangle (\ref{eq:diagram_def_dwyer_map}), though not exactly the same.

Notice that $i_0$ is a special Dwyer map. In particular, the category $P$ is a coreflective subcategory of $W$. Note that we use the language and notation of mapping cylinders mainly because it is common in the literature and because notation exists, although connection with mapping cylinders in \cite[§2.4]{WJR13} is interesting. Nevertheless, for the purpose of this argument, what matters is that $i_0$ is a sieve and has a retraction that is a right adjoint, which in this case is the projection $W\to P$ onto the first factor. Let $\bar{k} :NR\to T(f)$ denote the cobase change in $sSet$ of $k$ along $f$ and let $\bar{f}$ denote the cobase change in $sSet$ of $f$ along $k$. We will handle two cases.

We consider pairs $(x',y')$ of embedded simplices $x'$ and $y'$ of $T(f)$ that are siblings and that are of a fixed degree $q>0$. Notice that the relation \emph{being a sibling of} is an equivalence relation on the set of $q$-simplices. In the following, posets are viewed interchangeably as small categories and as a sets with a binary relation $\leq $ that is reflexive, antisymmetric and transitive. At a given moment in the argument, we adopt whichever viewpoint has the most convenient terminology.

The first case is when the common last vertex $x'\varepsilon _q=y'\varepsilon _q$ of the embedded siblings $x'$ and $y'$ is in the image of $\bar{k}$. In that case, $x'$ and $y'$ are in the image of $\bar{k}$ as it is an elysium. Two $q$-simplices of $NR$ whose images are $x'$ and $y'$, respectively, must be siblings. Any two siblings in the nerve of a poset are equal, so it follows that $x'=y'$ in this case. Thus $\eta (x')=\eta (y')$, trivially.

The second case, namely when $x'\varepsilon _q=y'\varepsilon _q$ is not in the image of $\bar{k}$, is highly non-trivial. We will handle this situation by inductively replacing the pair of siblings with another pair of siblings that are closer in a sense that we now make precise. Our induction has the following \emph{hypothesis}.

Suppose some integer $p<q$ is such that whenever two embedded siblings $x'$ and $y'$ of $T(f)$ whose common last vertex $x'\varepsilon _q=y'\varepsilon _q$ is not in the image of $\bar{k}$, then $x'$ has a sibling $z'$ and $y'$ has a sibling $w'$ with
\begin{displaymath}
\begin{array}{rcl}
\eta (x') & = & \eta (z') \\
\eta (y') & = & \eta (w')
\end{array}
\end{displaymath}
such that the unique simplices $z$ and $w$ of $NW$ with
\begin{displaymath}
\begin{array}{rcl}
z' & = & \bar{f} (z) \\
w' & = & \bar{f} (w)
\end{array}
\end{displaymath}
satisfy $z\varepsilon _j=w\varepsilon _j$ for each non-negative integer $j$ with $p<j\leq q$. The uniqueness of $z$ and $w$ comes from the fact that $\bar{f_q}$ is injective on the complement of $(NP)_q$ in $(NW)_q$. Note that $z'$ and $w'$ are siblings as $x'$ and $y'$ are.

Consider the event that $p=-1$. Then the simplices $z$ and $w$ of $NW$ are siblings. Therefore $z=w$ as $NW$ is the nerve of a poset. Hence $z'=w'$.

For the \emph{base step}, note that our induction hypothesis is satisfied for $p=q-1$. We will verify this in the next paragraph. Notice that the induction moves in the opposite direction, namely that the inductive step will verify that the hypothesis is true for $p-1$ whenever we know that it is true for $p$.

Recall that a simplex of $T(f)$ of any degree is exclusively and uniquely the image of either a simplex of $NR$ or a simplex of $NW$ that is not in the image of $k$. If $x'$ and $y'$ are embedded siblings whose last vertex $x'\varepsilon _q=y'\varepsilon _q$ is not in the image of $\bar{k}$, then the unique $q$-simplices $x$ and $y$ with
\begin{displaymath}
\begin{array}{rcl}
x' & = & \bar{f} (x) \\
y' & = & \bar{f} (y)
\end{array}
\end{displaymath}
are such that neither $x\varepsilon _q$ nor $y\varepsilon _q$ is in the image of $k$. These two $0$-simplices, in other words, reside in the back end of the cylinder $NW$, which is the image of $Ni_1$. We think of the back end as the nerve of the full subcategory $V$ of $W$ whose objects are those that are not in the image of $i_0$. In other words, the back end is the nerve of a cosieve, which is in this case the image of $i_1$.

The composite
\[NV\to NW\xrightarrow{\bar{f} } T(f)\to M(f)\]
is degreewise injective as it is the nerve of an injective map, hence
\[NV\to NW\xrightarrow{\bar{f} } T(f)\]
is degreewise injective. It follows that $x\varepsilon _q=y\varepsilon _q$.

Now we do the \emph{inductive step}. Take a pair $(x',y')$ of embedded $q$-simplices $x'$ and $y'$ of $T(f)$ that are siblings and whose common last vertex $x'\varepsilon _q=y'\varepsilon _q$ is not in the image of $\bar{k}$. Take a sibling $z''$ of $x'$ and a sibling $w''$ of $y'$ with
\begin{displaymath}
\begin{array}{rcl}
\eta (x') & = & \eta (z'') \\
\eta (y') & = & \eta (w'')
\end{array}
\end{displaymath}
and such that the unique simplices $z_2$ and $w_2$ of $NW$ with
\begin{displaymath}
\begin{array}{rcl}
z'' & = & \bar{f} (z_2) \\
w'' & = & \bar{f} (w_2)
\end{array}
\end{displaymath}
satisfy $z_2\varepsilon _j=w_2\varepsilon _j$ for each non-negative integer $j$ with $p<j\leq q$.

In the case when
\[z_2\varepsilon _p=w_2\varepsilon _p,\]
then we simply define
\begin{displaymath}
\begin{array}{rcl}
z' & = & z'' \\
z & = & z_2 \\
w' & = & w'' \\
w & = & w_2,
\end{array}
\end{displaymath}
and we are done.

Else if
\[z_2\varepsilon _p\neq w_2\varepsilon _p,\]
then there is work to be done.

Because the map $NV\to NW\xrightarrow{\bar{f} } T(f)$ is degreewise injective it follows that $z_2\varepsilon _p$ or $w_2\varepsilon _p$ resides in the front end of the cylinder $NW$, so $x''\varepsilon _p=y''\varepsilon _p$ is in the image of $\bar{k}$. The set $T(f)_0$ of $0$-simplices is the disjoint union of the image of $\bar{k} _0$ and the image under $\bar{f} _0$ of the complement of the image of $k_0$. In particular, both $z_2\varepsilon _p$ and $w_2\varepsilon _p$ reside in the front end of the cylinder, which is the image of $k$.

For the next piece of argument, we shift focus somewhat and view $z_2$ and $w_2$ as functors $[q]\to W$. Notice that, say the $0$-simplex $z_2\varepsilon _j$ in $NW$ corresponds to the object $z_2(j)$ in $W$ for each $j$. Combine the two functors $z_2$ and $w_2$ to form the solid arrow diagram
\begin{equation}
\label{eq:third_diagram_proof_prop_barratt_nerve_rep_map_dcr_inj}
\begin{gathered}
\xymatrix{
z_2(0) \ar[d] && w_2(0) \ar[d] \\
\dots \ar[d] && \dots \ar[d] \\
z_2(p-1) \ar[d] && w_2(p-1) \ar[d] \\
z_2(p) \ar[dr] \ar@{-->}[r] & z_2(p)\vee w_2(p) \ar@{-->}[d] & w_2(p) \ar[ld] \ar@{-->}[l] \\
& z_2(p+1)=w_2(p+1) \ar[d] \\
& \dots \ar[d] \\
& z_2(q)=w_2(q)
}
\end{gathered}
\end{equation}
in the category $W$. The diagram (\ref{eq:third_diagram_proof_prop_barratt_nerve_rep_map_dcr_inj}) looks like a \emph{zipper}. To realize this also reveals the idea behind the proof of \cref{prop:barratt_nerve_rep_map_dcr_inj}, which is to show that $\eta (x')$ and $\eta (y')$ are equal by performing a zipping in the category $W$.

Think of $W$ as embedded in $Q=\Delta [n+1]^\sharp$ as in (\ref{eq:zeroth_diagram_proof_lem_second_reduction}) except that $n$ is replaced by $n+1$. The category $Q$ has the property that whenever there is a cocone on a diagram
\begin{displaymath}
\xymatrix{%
q \ar@(ld,lu)^{id} \\
q' \ar@(ld,lu)^{id}
}
\end{displaymath}
in $Q$, then there is a universal such, or in other words a coproduct of $q$ and $q'$. The coproduct in a poset of two objects is often referred to as the join of the two objects. Frequently, the symbol $\vee$ denotes the join operation so that the join of $q$ and $q'$ is denoted $q\vee q'$.

The category $W$ is obtained from $Q$ by just removing the object $\varepsilon _2:[0]\to [2]$ given by $0\mapsto 2$ and each morphism whose source is $\varepsilon _2$. It follows that the category $W$ inherits the property from $Q$ that was described in the previous paragraph, namely that the existence of a cocone implies the existence of a join. Because $P$ is a coreflective subcategory of $W$, the join in $W$ of $z_2(p)$ and $w_2(p)$ is an object of $P$.

Notice that there are two obvious $(q+1)$-simplices in $NW$ that appear in (\ref{eq:third_diagram_proof_prop_barratt_nerve_rep_map_dcr_inj}), namely
\[z_2(0)\to \dots \to z_2(p)\to z_2(p)\vee w_2(p)\to z_2(p+1)\to \dots \to z_2(q)\]
denoted $\tilde{z}$ and
\[w_2(0)\to \dots \to w_2(p)\to z_2(p)\vee w_2(p)\to w_2(p+1)\to \dots \to w_2(q)\]
denoted $\tilde{w}$. We have an application in mind for them, which will become clear shortly if it has not already.

Because $P$ is a sieve in $W$, the subdiagram
\begin{displaymath}
\xymatrix{
z_2(0) \ar[d] && w_2(0) \ar[d] \\
\dots \ar[d] && \dots \ar[d] \\
z_2(p-1) \ar[d] \ar[dr] && w_2(p-1) \ar[ld] \ar[d] \\
z_2(p) \ar[r] & z_2(p)\vee w_2(p) & w_2(p) \ar[l] \\
}
\end{displaymath}
in $W$ of the big diagram above is really a diagram in $P$, whereas the object $z_2(q)=w_2(q)$ is not an object of $P$.

Notice that $\varphi (z_2(p))=\varphi (w_2(p))$ due to the fact that $z''$ and $w''$ are siblings, which in particular implies that $z''\varepsilon _p=w''\varepsilon _p$. This is because $\varphi$ is defined as $\varphi =(\bar{r} )^\sharp$ where $r$ is from \cref{prop:barratt_nerve_rep_map_dcr_inj}. If we can prove that
\begin{equation}\label{equation_dsd^2=bsd}
\varphi (z_2(p)\vee w_2(p))=\varphi (z_2(p)),
\end{equation}
which we can, then the two simplices $\tilde{z}$ and $\tilde{w}$ give rise to simplices in $T(f)$ that become degenerate under desingularization (in a specific way).

Let $z$ denote the simplex
\[z_2(0)\to \dots \to z_2(p-1)\to z_2(p)\vee w_2(p)\to z_2(p+1)\to \dots \to z_2(q)\]
in $NW$ and $z'$ its image under $\bar{f}$. When we verify (\ref{equation_dsd^2=bsd}) it will follow that $z'$ and $z''$ are siblings. By assumption, the simplex $z''$ is a sibling of $x'$. It will thus follow that $x'$ is a sibling of $z'$ as \emph{being a sibling of} is an equivalence relation. Moreover, the image $\bar{f} (\tilde{z} )$ has the property that
\[\bar{f} (\tilde{z} )\varepsilon _p=\bar{f} (\tilde{z} )\varepsilon _{p+1}.\]
This means that $\bar{f} (\tilde{z} )$ becomes degenerate under desingularization. More precisely, we get that $\eta \bar{f} (\tilde{z} )$ splits off the degeneracy operator $\sigma _p$. In other words, the simplices $x'$ and $z'$ become identified under desingularization, meaning $\eta (x')=\eta (z')$.

Similarly, let $w$ denote the simplex
\[w_2(0)\to \dots \to w_2(p-1)\to z_2(p)\vee w_2(p)\to w_2(p+1)\to \dots \to w_2(q)\]
in $NW$ and $w'$ its image under $\bar{f}$. Then $w'$ and $w''$ are siblings if (\ref{equation_dsd^2=bsd}) holds. By assumption, the simplex $w''$ is a sibling of $y'$. It will thus follow that $y'$ is a sibling of $w'$. We get that $\eta (y')=\eta (w')$ as $\eta \bar{f} (\tilde{w} )$ splits off the elementary degeneracy operator $\sigma _p$.

Note that the equations
\begin{displaymath}
\begin{array}{rcl}
z(p) & = & w(p) \\
& \dots \\
z(q) & = & w(q)
\end{array}
\end{displaymath}
hold by definition of $z$ and $w$. This means that verifying (\ref{equation_dsd^2=bsd}) finishes the induction step in the case when $z_2\varepsilon _p\neq w_2\varepsilon _p$.

We go on to verify (\ref{equation_dsd^2=bsd}). It could be that $w_2(p)$ is a face of $z_2(p)$, meaning $z_2(p)\vee w_2(p)=z_2(p)$. Similarly, it could be that $z_2(p)$ is a face of $w_2(p)$, meaning $z_2(p)\vee w_2(p)=w_2(p)$. In both cases, we trivially obtain (\ref{equation_dsd^2=bsd}). Let us consider the non-trivial case when neither one is a face of the other.

Notice that if $q$ and $q'$ are objects of $Q=\Delta [n+1]^\sharp$ whose join $q\vee q'$ exists, then the face operator $q\vee q'$ is the one whose image is the union of the images of $q$ and $q'$. This operation is inherited by the subcategory $W$ of $Q$ as was pointed out earlier. There are unique face operators $\mu$ and $\nu$ such that
\begin{displaymath}
\begin{array}{rcl}
z_2(p) & = & (z_2(p)\vee w_2(p))\mu \\
w_2(p) & = & (z_2(p)\vee w_2(p))\nu .
\end{array}
\end{displaymath}
The union of the images of $\mu$ and $\nu$ is equal to their common target. Also, neither image is contained in the other because we now consider the non-trivial case when neither of the simplices $z_2(p)$ and $w_2(p)$ is a face of the other.

Consider applying \cref{prop:deflation_theorem} in the case when $y=\bar{r} (z_2(p)\vee w_2(p) )$. Recall that $\varphi =(\bar{r} )^\sharp$. We get that
\[\varphi (z_2(p)\vee w_2(p))=y^\sharp \]
by definition of $\varphi$ and we can let $\mu$ and $\nu$ denote the face operators that applied to $z_2(p)\vee w_2(p)$ yield $z_2(p)$ and $w_2(p)$, respectively.

Furthermore,
\begin{equation}
\label{eq:motivating_lemma_property_of_regular_simplicial_set}
\begin{gathered}
\begin{array}{rcl}
\varphi (z_2(p)) & = & \varphi ((z_2(p)\vee w_2(p))\mu ) \\
& = & (\bar{r} )^\sharp ((z_2(p)\vee w_2(p))\mu ) \\
& = & (\bar{r} ((z_2(p)\vee w_2(p))\mu ))^\sharp \\
& = & (\bar{r} ((z_2(p)\vee w_2(p)))\mu )^\sharp \\
& = & (y\mu )^\sharp
\end{array}
\end{gathered}
\end{equation}
and similarly $\varphi (w_2(p))=(y\nu )^\sharp$. The equation (\ref{equation_dsd^2=bsd}) follows from \cref{prop:deflation_theorem}.

From the verification of (\ref{equation_dsd^2=bsd}), it follows that the sibling $z'$ of $x'$ and the sibling $w'$ of $y'$ are such that
\begin{displaymath}
\begin{array}{rcl}
\eta (x') & = & \eta (z') \\
\eta (y') & = & \eta (w')
\end{array}
\end{displaymath}
and such that the pair $(z,w)$ of simplices $z$ and $w$ of $NW$ with
\begin{displaymath}
\begin{array}{rcl}
z' & = & \bar{f} (z) \\
w' & = & \bar{f} (w)
\end{array}
\end{displaymath}
has the property that $z\varepsilon _j=w\varepsilon _j$ for each non-negative integer $j$ with $p-1<j\leq q$. This means that having verified (\ref{equation_dsd^2=bsd}) finishes the induction step in the case when $z_2\varepsilon _p\neq w_2\varepsilon _p$. Thus the map $\eta _{T(f)}$ takes each pair of embedded siblings of degree $q$ to the same simplex.

As the integer $q>0$ was arbitrary, the conclusion holds for each positive integer. Namely that $\eta _{T(f)}$ takes each pair of embedded siblings to the same simplex. Recall that $f=B(\bar{r} )$. We are ready to prove \cref{prop:barratt_nerve_rep_map_dcr_inj}.
\begin{proof}[Proof of \cref{prop:barratt_nerve_rep_map_dcr_inj}.]
We have just proven by induction on what we may call the \emph{proximity of a pair of siblings} that $\eta _{T(B(\bar{r} ))}$ takes each pair of embedded siblings of degree $q$ to the same simplex, for each $q>0$. This is trivially true for $q=0$ as well, though irrelevant.

The simplicial set $DT(B(\bar{r} ))$ is non-singular, the simplicial set $M(B(\bar{r} ))$ is the nerve of a poset and $\eta _{T(B(\bar{r} ))}$ is degreewise surjective. Furthermore, the map
\[cr:T(B(\bar{r} )\to M(\bar{r} )\]
is injective in degree $0$ by \cref{ex:pushout_poset_along_dwyer}. Thus \cref{prop:criterion_degreewise_injective_into_nerve_computation} is applicable to (\ref{eq:first_diagram_proof_prop_barratt_nerve_rep_map_dcr_inj}).

By \cref{prop:criterion_degreewise_injective_into_nerve_computation}, the map
\[dcr:DT(B(\bar{r} )\to M(\bar{r} ))\]
is injective in each positive degree.
\end{proof}

\section{Comparison of mapping cylinders}
\label{sec:comparison}

Recall from \cref{thm:barratt_nerve_rep_map_dcr_iso} that we consider a regular simplicial set $X$ and an arbitrary simplex $y$ of $X$, say of degree $n$. The theorem makes the claim that
\[dcr:DT(B(\bar{y} ))\xrightarrow{\cong } M(B(\bar{y} ))\]
is an isomorphism, which we will now prove.
\begin{proof}[Proof of \cref{thm:barratt_nerve_rep_map_dcr_iso}.]
First, we argue that $dcr$ is bijective in degree $0$. Consider \cref{ex:pushout_poset_along_dwyer} in the case when the map $\varphi :P\to R$ is the map
\[(\bar{y} )^\sharp :\Delta [n]^\sharp \to X^\sharp\]
and when $P\to Q$ is the map
\[i_0:\Delta [n]^\sharp \to \Delta [n]^\sharp \times [1].\]
Then it follows directly from \cref{ex:pushout_poset_along_dwyer} that the cylinder reduction map
\[T(B(\bar{y} ))=NQ\sqcup _{NP}NR\xrightarrow{cr} N(Q\sqcup _PR)=M(B(\bar{y} ))\]
is bijective in degree $0$. As
\[\eta _{T(B(\bar{y} ))}:T(B(\bar{y} ))\to DT(B(\bar{y} ))\]
is degreewise surjective it follows that
\[dcr:DT(B(\bar{y} ))\to M(B(\bar{y} ))\]
is bijective in degree $0$. Recall that these three maps fit into the commutative triangle (\ref{eq:first_diagram_proof_prop_barratt_nerve_rep_map_dcr_inj}).

Next, we argue that $dcr$ is degreewise surjective. Let $Y$ denote the image of $\bar{y} :\Delta [n]\to X$. Then $BY$ is the image of $B(\bar{y} )$ \cite[Lem.~2.4.20]{WJR13}. Consider the diagram
\begin{equation}
\label{eq:first_diagram_proof_thm_barratt_nerve_rep_map_dcr_iso}
\begin{gathered}
\xymatrix{
B(\Delta [n]) \ar[d] \ar[r] & BY \ar[d] \ar[r] & BX \ar[d] \\
B(\Delta [n])\times \Delta [1] \ar[d] \ar[r] & DT \ar[d]^{dcr} \ar[r] & DT(B(\bar{y} )) \ar[d]^{dcr} \\
B(\Delta [n])\times \Delta [1] \ar[r] & M \ar[r] & M(B(\bar{y} ))
}
\end{gathered}
\end{equation}
where $T$ denotes the topological mapping cylinder of the corestriction of $B(\bar{y} )$ to its image $BY$ and where $M$ denotes the reduced mapping cylinder of the same map.

It follows from \cref{prop:map_between_regular_reduction_map_simple} that $dcr:DT\to M$ is degreewise surjective. This is because both $\Delta [n]$ and $Y$ are finite regular simplicial sets. We will explain that
\[dcr:DT(B(\bar{y} ))\to M(B(\bar{y} ))\]
is the cobase change in $sSet$ of $DT\to M$ along $BY\to BX$. Thus we obtain the desired result.

Note that
\[B(\Delta [n])\times \Delta [1]\to DT\]
is the cobase change in $nsSet$ of $B(\Delta [n])\to BY$ along
\[B(\Delta [n])\to B(\Delta [n])\times \Delta [1].\]
Furthermore, the map
\[B(\Delta [n])\times \Delta [1]\to DT(B(\bar{y} ))\]
is the cobase change in $nsSet$ of $B(\Delta [n])\to BX$ along
\[B(\Delta [n])\to B(\Delta [n])\times \Delta [1].\]
Consequently, the map
\[DT\to DT(B(\bar{y} ))\]
is the cobase change in $nsSet$ of $BY\to BX$ along $BY\to DT$.

The map $BY\to M$ is degreewise injective, hence $BY\to DT$ is degreewise injective. As $nsSet$ is a reflective subcategory of $sSet$, it follows that the map
\[DT\to DT(B(\bar{y} ))\]
is even the cobase change in $sSet$ of $BY\to BX$ along $BY\to DT$.

Next, consider the diagram
\begin{equation}
\label{eq:second_diagram_proof_thm_barratt_nerve_rep_map_dcr_iso}
\begin{gathered}
\xymatrix{
\Delta [n]^\sharp \ar[d] \ar[r] & Y^\sharp \ar[d] \ar[r] & X^\sharp \ar[d] \\
\Delta [n]^\sharp \times [1] \ar[r] & \Delta [n]^\sharp \times [1]\sqcup _{\Delta [n]^\sharp }Y^\sharp \ar[r] & \Delta [n]^\sharp \times [1]\sqcup _{\Delta [n]^\sharp }X^\sharp
}
\end{gathered}
\end{equation}
in $PoSet$. Remember that $B=NU(-)^\sharp$. The cocontinous functor
\[(-)^\sharp :sSet\to PoSet\]
turns degreewise injective maps into sieves. A cobase change in $PoSet$ of a sieve is again a sieve, so $Y^\sharp \to \Delta [n]^\sharp \times [1]\sqcup _{\Delta [n]^\sharp }Y^\sharp$ is a sieve. The right hand square of (\ref{eq:second_diagram_proof_thm_barratt_nerve_rep_map_dcr_iso}) is a cocartesian square that is preserved under $U:PoSet\to Cat$. This is because both legs are sieves, which means that the pushout in $Cat$ is a poset and because $PoSet$ is a reflective subcategory of $Cat$.

It is even true that $M\to M(B(\bar{y} ))$ is the cobase change in $sSet$ of $BY\to BX$ along $BY\to M$ as $N:Cat\to sSet$ preserves a cocartesian square in $Cat$ whenever both legs are sieves.

As a result of the considerations above, we see from (\ref{eq:first_diagram_proof_thm_barratt_nerve_rep_map_dcr_iso}) that
\[dcr:DT(B(\bar{y} ))\to M(B(\bar{y} ))\]
is the cobase change in $sSet$ of $DT\to M$ along $BY\to BX$, which is the desired result.

Finally, the map
\[dcr:DT(B(\bar{y} ))\to M(B(\bar{y} ))\]
is degreewise injective in degrees above $0$, for this is precisely what \cref{prop:barratt_nerve_rep_map_dcr_inj} says.

The map $dcr$ is thus seen to be bijective in degree $0$, it is degreewise surjective and it is injective in degrees above $0$. This concludes the proof that $dcr$ is an isomorphism.
\end{proof}
\noindent The proof of \cref{thm:barratt_nerve_rep_map_dcr_iso} was the last piece of the proof of our main result, which is \cref{thm:main_opt_triang}.

\bibliographystyle{unsrt}  
\bibliography{main}

\end{document}